\documentclass[reqno,10pt]{amsart}
\usepackage{amsfonts,amssymb,amsthm,amsbsy,latexsym,amscd,amsmath,dsfont,euscript, enumerate, geometry,graphicx,ifthen, manfnt, marvosym, verbatim, calc,mathrsfs,marvosym,textcomp,todonotes,tikz,tikz-cd, color,xcolor,setspace,stmaryrd,upgreek, bm}
\usetikzlibrary{positioning, shapes,calc}
\usepackage{graphicx}
\usepackage{circuitikz}
\definecolor{bulgarianrose}{rgb}{0.28, 0.02, 0.03}
\usepackage[linkcolor=bulgarianrose,citecolor=bulgarianrose,colorlinks=true,hypertexnames=false]{hyperref}

\newtheorem{theorem}{Theorem}[section]
\newtheorem{corollary}[theorem]{Corollary}
\newtheorem{lemma}[theorem]{Lemma}

\theoremstyle{definition}

\newtheorem{construction}[theorem]{Construction}

\newtheorem*{remark}{Remark}

\newtheorem*{acknowledgement}{Acknowledgement}

\newcommand{\Ex}{\ensuremath{{\textnormal{E}}}}

\makeatletter
\def\imod#1{\allowbreak\mkern10mu({\operator@font mod}\,\,#1)}
\def\@textbottom{\vskip\z@\@plus 18pt}
\let\@texttop\relax
\makeatother

\title[Ramsey number involving  bipartite graphs and odd wheel]{An exact Ramsey number of \\ 
large bipartite graphs versus odd wheel}
\author{Sayan Gupta}
\author{Kaushik Majumder}

\address{\newline 
\newline (a) School of Mathematical Sciences\\ \newline National Institute of Science Education and Research (NISER) Bhubaneswar\\\newline Jatni, Khurda-$752050$, Odisha, India.
\newline (b) Homi Bhabha National Institute (HBNI)\\ \newline Training School Complex, Anushakti Nagar, Mumbai- $400094$, India.
\newline \textnormal{\textestimated-Mails}: (Sayan Gupta) {\tt sayan.gupta\MVAt niser.ac.in, sayangupta4u\MVAt gmail.com}
\newline \textnormal{\textestimated-Mails}: (Kaushik Majumder) {\tt kaushikmajumder\MVAt niser.ac.in, kaushikbnmajumder\MVAt gmail.com}}


\begin{document}
\thispagestyle{empty}
\vspace{10cm}

\begin{abstract}
The Ramsey number for the pair of graphs $\mathbb{K}_{1,n}$ (star) versus $W_{m}$ (wheel) has been extensively studied. In contrast, the Ramsey number of $\mathbb{K}_{2,n}$ versus the wheel is not yet explored due to the bit more structural complexity of $\mathbb{K}_{2,n}$ compared to the star. In this article, we have established an exact value of $\mathbb{K}_{2,n}$ versus $W_{m}$ for large $n$ and $m$. In particular, we have proved 
\begin{equation*}
R(\mathbb{K}_{2,n}, W_{m})=3n+4,
\end{equation*}
whenever $n$ and $m$ are sufficiently large integers satisfying $n\geq4m$ and $m$ is an odd integer. This proves the $W_{m}$-goodness of $\mathbb{K}_{2,n}$. Our proof combines probabilistic methods with an analysis of structural dependencies. As part of the argument, we resolve a structural rigidity question concerning highly dependent neighbourhoods (Lemma~\ref{nbd conservation 2}). 
\end{abstract}

\maketitle

\section{Introduction}

Let $G$ and $H$ be two graphs. The \emph{Ramsey number} $R(G, H)$ is the minimum positive integer $N$ such that for each graph $\Gamma$ on $N$ vertices either $\Gamma$ contains a copy of $G$ or its complement $\overline{\Gamma}$ contains a copy of $H$ as a subgraph. For a vertex $v$ in $G$, we denote the set of neighbours of $v$ in $G$ as $N_{G}(v)$ and its size $|N_{G}(v)|$ is called degree with textbook notation $\deg_{G}(v)$. In the case of a bipartite graph $G$ with bipartition $A\sqcup B$, if $v\in A$, then we simply denote the neighbours of it by $N_{B}(v)$ as it does not contain any neighbour in $A$. We denote the maximum and minimum degrees of a graph $G$ as $\Delta(G)$ and $\delta(G)$, respectively.
For a graph $G$, suppose its vertex set $V$ is decomposed into two sets $P\sqcup Q$, such that the ratio $\frac{|P|}{|V|}$ is sufficiently small enough, in the sense that such ratio $\frac{|P|}{|V|}$ tends to zero, whenever $|V|$ tends to $\infty$. In this setting, we may regard $Q$ as a \emph{dense} or \emph{regular} set in $V$, while $P$ behaves like a \emph{null} or \emph{irregular} set. The presence of a null set $P$ complicates the analysis of the dependency tree. However, the dominant contribution comes from the dense set $Q$, which plays a central role in analysis. To address this, we consider the induced subgraph by the dense set $Q$, denoted by the textbook notation $G[Q]$. The foundational local dependency relation acts as the primary anchor for our analysis. For each $x\in Q$, the neighbourhood $N_{G}(x)$ is bounded by  
\begin{equation}\label{nbd dependency}
N_{G[Q]}(x)\subset N_{G}(x)\subset N_{G[Q]}(x)\sqcup P.
\end{equation}
This inclusion facilitates a more transparent examination of the structural properties within the relevant subgraphs. The \emph{dependent random choice technique} is built precisely to manage such local dependency structures. The term \emph{dependent} refers to the fact that the choice of an arbitrarily chosen neighbourhood of a vertex is not independent across vertices. The neighbourhoods of different vertices overlap in highly structured ways. By \emph{choosing} vertices according to a \emph{random} experiment (with or without replacement) repeatedly, one tends to suppress the influence of: (a) vertices of very low degree (since they are rarely hit), and (b) vertices of very high degree (since they contribute disproportionately little to intersection patterns). These two observations trigger a natural decomposition of the vertex set $V=P\sqcup Q$. In this framework, the set $P$ represents the \emph{irregular} or \emph{null} component—the extreme outliers that can be safely ignored. The analysis then examines $Q$, where the structural properties are determined by the behaviour of the mediocre degree. This `cleaning' effect serves as the core mechanism of dependent random choice and is further facilitated by the inclusion of local dependencies. We implement this specific random experiment within the proof of Lemma~\ref{intersection lemma}.

Burr proved a general lower bound on the Ramsey number for the pair of graphs $(G,H)$. His statement is the following, which reveals the inequality relation between $R(G,H)$ and the parameters $|V(G)|$, $\chi(H)$ and $\sigma(H)$. Here, the parameter $\chi(H)$ denotes the textbook notation of the chromatic number of $H$ and $\sigma(H)$ denotes the size of the smallest colour class in a $\chi(H)$-colouring of $H$.
\begin{theorem}\cite[Theorem~1]{MR635872}
Let $G$ be a connected graph and $H$ be a graph satisfying $|V(G)|\geq \sigma(H)$, then 
 \begin{equation}\label{Burr inequality}
R(G,H)\geq (|V(G)|-1)(\chi(H)-1)+\sigma(H). 
\end{equation}
\end{theorem}

\noindent We say that a graph $G$ is \emph{$H$-good} if equality holds in Burr’s lower bound \eqref{Burr inequality} for the pair $(G,H)$. This notion of goodness, introduced by Burr and Erd\H{o}s in \cite{MR693019}. Generally speaking, it identifies the graph pairs $(G,H)$ for which Burr’s standard lower bound is, in fact, tight. The study of Ramsey goodness has attracted significant attention, particularly in cases where $H$ attains a bounded chromatic number independent of the graph order. Determining which graphs are $H-$good for various families of graphs $H$, and understanding the structural reasons behind goodness, remains one of the central and challenging problems in graph Ramsey theory. Over the past several decades, substantial progress has been made for trees, cycles, and various sparse or nearly bipartite graphs. For some interesting results on the Ramsey goodness, one can refer to \cite{MR3071850,MR4624632,MR3575209,MR4142765,MR2520282}.
  
A well-studied example in the literature on Ramsey goodness is the pair $(\mathbb{K}_{1,n}, W_{m})$, where $\mathbb{K}_{1,n}$ denotes the star and $W_{m}$ is the wheel graph on $m+1$ vertices. The formal definition of the \emph{wheel} $W_{m}$ on $m+1$ vertices is the graph $\mathbb{K}_{1} \circledast C_{m}$. Here, the notation $\circledast$ indicates the join of two graphs, that is, $\mathbb{K}_{1}$ is adjacent to all vertices of the cycle $C_{m}$. Recall that for each integer $m$, $\chi(W_{m})\in\{3,4\}$ and $\sigma(W_{m})=1$. The general Ramsey number for the pair $(\mathbb{K}_{1,n}, W_{m})$ in the case of odd wheels was first determined in \cite{MR2083456} for $n \geq m-1$, establishing that $\mathbb{K}_{1,n}$ is $W_{m}$-good in this range. This bound was later improved to $n \geq \frac{m+1}{2}$ in \cite{MR2186318}. For further results on star–wheel Ramsey numbers, see \cite{MR3461979}.

\begin{theorem}\cite[Theorem~2]{MR2083456}
If $n\geq m-1$ and $m\geq 3$ is an odd integer, then 
\begin{equation*}
R(\mathbb{K}_{1,n},W_{m})=3n+1.
\end{equation*}
\end{theorem}

A \emph{book} graph is the defined as $\mathbb{K}_{2}\circledast \overline{\mathbb{K}_{n}}$, i.e., $n$ triangles sharing a common edge. It is denoted by $B_{n}$. An exact value of the book-wheel Ramsey number is the following.
\begin{theorem}\cite[Theorem~1]{MR1781297}\label{zhou}
If $m\geq 5n+3$ and $n\geq 1$, then
\begin{equation*}
 R(W_{m},B_{n})= 2m+1.
\end{equation*}
\end{theorem}
\noindent This result proves that $W_{m}$ is $B_{n}$-good. The motivating observation is that the graph $\mathbb{K}_{2,n}$ is sandwiched between the star and book graphs. In particular,
\begin{equation*}
\mathbb{K}_{1,n}\subset \mathbb{K}_{2,n}\subset B_{n}. 
\end{equation*}
Here, $H\subset G$ means a copy of the graph $H$ is contained in $G$ as a subgraph. A graph $G$ is $ \mathbb{K}_{1,n}$-free if and only if the maximum degree $\Delta(G)\leq n-1$. As a consequence, we get a lower bound on the minimum degree of $\overline{G}$. This helps to identify the embedding of particular structures inside $\overline{G}$. Whereas in the case of $\mathbb{K}_{2,n}$, we do not have that particular advantage due to its structural complexity and dependency. Hence, the Ramsey goodness of $\mathbb{K}_{2,n}$ versus other graphs has not been studied yet in comparison to the star. The $C_{m}$-goodness of $\mathbb{K}_{2,n}$ for odd $m$ has been established as follows:-  
\begin{theorem}\label{Our result}\cite[Theorem~1.8]{gupta2025study}
For each odd integer $m\geq 7$ and for each integer $n\geq2m+499$ and $n\geq 3493$, then 
\begin{equation*}
R(\mathbb{K}_{2,n},C_{m})= 2n+3. 
\end{equation*}
\end{theorem}
\noindent However, in the case of the $\mathbb{K}_{2,n}$ graph versus wheel graph, there is no such explicit study till now. Although we get a direct consequence of $\mathbb{K}_{2,n}$-goodness of $W_{m}$ from Theorem~\ref{zhou} as $ R(W_{m},\mathbb{K}_{2,n})\leq  R(W_{m},B_{n})$. Now it is an interesting question to ask whether $\mathbb{K}_{2,n}$ is $W_{m}$-good or not. In this article, we answer this question in the affirmative for large $n$ and $m$.

\begin{theorem}\label{main theorem}
If $n$ and $m$ are sufficiently large positive integers satisfying $n\geq 4m$, where $m$ is an odd integer, then
\begin{equation*} 
R(\mathbb{K}_{2,n}, W_{m})=3n+4.
\end{equation*}
\end{theorem}
\begin{remark}
The precise largeness requirements on $n$ and $m$ are addressed in detail in Construction~\ref{construction}. The reader is referred to \eqref{large integer construction} for further clarification.
\end{remark}

It is quite challenging to analyse a $\mathbb{K}_{2,n}$-free graph $G$. Unlike the situation for star-free graphs, we do not obtain any immediate information about the maximum degree of $G$. Nevertheless, two basic observations concerning the neighbourhoods of each pair  $u,v\in V(G)$ are staring us in the face. The first is the identity
\begin{equation}\label{nbd complementary dependency}
\{u\}\sqcup N_{G}(u)\sqcup N_{\overline{G}}(u)=V(G),
\end{equation}
which holds for each graph $G$. The second observation, which holds for a $\mathbb{K}_{2,n}$-free graph $G$ and plays a crucial role in our argument, is the necessary and sufficient condition for $\mathbb{K}_{2,n}$-freeness, namely the bound
\begin{equation}\label{K_2n freeness dependency}
|N_{G}(u)\cap N_{G}(v)|\leq n-1
\end{equation}
for each pair of distinct vertices $u,v\in V(G)$. These observations allow us to extract a non-trivial lower bound on the minimum degree of the induced subgraph $\overline{G}[V(G)-\{x\}\,]$ for a suitable choice of $x\in V(G)$.

In certain cases, this information alone is not sufficient to guarantee the embedding of the required structure inside $\overline{G}$. To address this issue, our proof repeatedly applies the \emph{first moment method}, starting with the derivation of an upper bound on $\delta(G)$. The intuition is that if $G$ has a sufficiently large minimum degree, then $G$ must contain a copy of 
$\mathbb{K}_{2,n}$. The probabilistic aspect of this argument is explained in detail within the proof, rendering the exposition self-contained. In particular, we devote substantial effort to proving a self-contained structural rigidity question concerning highly dependent neighbourhoods (Lemma~\ref{nbd conservation 2}) within our exposition.


\section{Paraphernalia: tracing an arc from the literature}

The length of the smallest cycle in the graph is called the \emph{girth}. The \emph{circumference} is the length of the largest cycle. The \emph{ even circumference} and \emph{odd circumference} of a graph $G$ are the length of the largest even and odd cycle subgraph, respectively. In this article, we denote them by $\mathop{ec}(G)$ and $\mathop{oc}(G)$ respectively. A graph $G$ of order $n$ is pancyclic if it contains cycles of all lengths between $3$ and $n$. A \emph{weakly pancyclic} graph is a graph containing cycles of all lengths between its girth and circumference. 

The first step of the proof is to show that a $\mathbb{K}_{2,n}$-free graph $G$ on $3n+4$ vertices contains a vertex $v$ with degree sufficiently small. Consequently, this vertex contains a sufficiently large neighbourhood $\overline{H}$ (say) in $\overline{G}$. We aim to show the existence of a $C_{m}$ inside that large neighbourhood. Together with $v$, we construct a copy of $W_{m}$, which leads to a contradiction. In order to show this, our first goal is to prove that $\overline{H}$ is non-bipartite as our $m$ is odd. We have explained the proof in detail in a later section. After this, we require a sufficient condition on the minimum degree of $\overline{H}$ to guarantee the existence of the $C_{m}$. A few results related to pancyclic and weakly pancyclic graphs are beneficial in this regard. However, in our case, we observe that the degree conditions are insufficient. Hence, we need the following two results to prove the existence of a $C_{m}$ inside $\overline{H}$. 

\begin{lemma}\label{cycle lemma 1}\cite{MR1005319}
If $G$ forms a non-bipartite graph with $\kappa(G)\geq2$, $\delta(G)\geq r\geq3$ and $|V(G)|\geq 2r+1$, then $\mathop{ec}(G)\geq 2r$ and $\mathop{oc}(G)\geq 2r-1$. 
\end{lemma}

\noindent In the above statement $\kappa(\Gamma)$ denotes the \emph{connectivity parameter} to deal the connectivity issues of the graph $\Gamma$. If $V(\Gamma)$ admits a partition $R\sqcup S$, where the removal of $R$ produces an induced subgraph $\Gamma[S]$ that is not connected. Thus such parameter, $\kappa(\Gamma)$ formally defined as,
\begin{equation*}
\kappa(\Gamma)=\min\left\{|R|: V(\Gamma)=R\sqcup S \textup{ and }\Gamma[S]\textup{  is not connected }\right\}.
\end{equation*}
In our setup, we are concerned only with the odd circumference. Accordingly, we construct a non-bipartite graph $\Gamma$ satisfying $\kappa(\Gamma)\geq2$, and we choose $r\geq3$  such that if $\delta(\Gamma)\geq r$ holds then $|V(\Gamma)|\geq2r+1$ holds. In this article, we take $r$ to be a suitable fraction of $V(\Gamma)$, for instance  $r=\frac{|V(\Gamma)|}{10}$ or $r=\frac{|V(\Gamma)|}{7}$ or $r=\frac{|V(\Gamma)|}{5}$. Then using Lemma~\ref{cycle lemma 1}, we conclude that such a graph $\Gamma$ contains an odd cycle of length at least $2r-1$.

\begin{lemma}\label{cycle lemma 2}\cite[Theorem~1]{MR1939069}
For each $\varepsilon>0$, there exists a real number $K=\frac{75\cdot10^{4}}{\varepsilon^{5}}=K(\varepsilon)$ (say) such that if $N$ is an integer with $N\geq \frac{45K}{\varepsilon^{4}}$ and $G$ is a graph with $\frac{\delta(G)}{|V(G)|}\geq\varepsilon$, where $|V(G)|\geq N$, then such $G$ contains a copy of cycle with length $m$, for each even integer $m\in [4, \mathop{ec(G)}-K]$ and for each odd integer $m\in [K, \mathop{oc}(G)-K]$. 
\end{lemma}

Another key lemma used in our proof is Lemma 7 of \cite{MR4597476}. To apply Lemma \ref{cycle lemma 1} to a graph $G$, we require a sufficient condition $\kappa(G)\geq2$. Accordingly, we decompose the graph $\overline{H}$ into disjoint components such that ``almost" each component satisfies this requirement. We restate Lemma 7 of \cite{MR4597476} below; its notably concise proof appears in the cited work. The underlying argument relies on an iterative procedure together with a construction directed by the connectivity parameter. The central observation driving this lemma is that ``$r<k$", which follows from \eqref{nbd dependency}. 

\begin{lemma}\label{Star-cycle}
For each graph $G$ with $\delta(G)\geq k+\frac{|V(G)|}{k}$, where $k\in[2,|V(G)|]$, then vertex set $V(G)$ admits a decomposition 
\begin{equation*}
 R\sqcup S_{1}\sqcup\ldots\sqcup S_{r},   
\end{equation*}
where $|R|\leq r-1$, $r<k$, and the induced subgraph $G[S]$ satisfies $\kappa(G[S])\geq2$,  for each $S\in\{S_{1},\ldots,S_{r}\}$.
\end{lemma}

The following lemma plays a key role in our proof, as it shows that the neighbourhood of the maximum-degree vertex in the graph we study is non-bipartite. This idea serves as a guiding principle of this article. 
\begin{lemma}\label{intersection lemma}
Let $G$ be a bipartite graph with a bipartition $A\sqcup B$ and $|N_{B}(a)|\geq d$ for each $a\in A$. Then there exist at least two vertices $x,y\in A$ such that 
\begin{equation*}
 |N_{B}(x)\cap N_{B}(y)|> \frac{d^{2}}{|B|}-\frac{d}{|A|}.   
\end{equation*}
\end{lemma}
\begin{proof}
We proceed with the proof using the first moment method. We perform a random experiment of choosing two vertices one by one without replacement from $A$. In this case, the sample space equals $\Omega=\{(x,y): x,y\in A; x\neq y\}$, which has size $|\Omega|=|A|(|A|-1)$. This yields the probability space $(\Omega, 2^{\Omega},\Pr)$,  where each singleton sample assigned equal probability, i.e. $\Pr(\{\omega\})=\frac{1}{|\Omega|}$ for each $\omega\in\Omega$. We further construct the random variable $X:\Omega\rightarrow \mathbb{R}$, where for each $(x,y)\in\Omega$

 \begin{equation*}
X(x,y)=|N_{B}(x)\cap N_{B}(y)|.
\end{equation*}
For a vertex $u\in B$, we construct the random variable $X_{u}:\Omega\rightarrow \mathbb{R}$, where for each $(x,y)\in\Omega$ 
\begin{align*}
X_{u}(x,y)&=\left\{\begin{array}{lll}
			1& \textnormal{ if  }  u\in N_{B}(x)\cap N_{B}(y)\\
			0& \textnormal{ else }&  
		\end{array}\right..
\end{align*}
Here, $X_{u}$ forms a Bernoulli random variable with parameter $\Pr(X_{u}=1)$, where 
\begin{equation*}
\Ex[X_{u}]= \Pr(X_{u}=1)= \frac{|N_{A}(u)|(|N_{A}(u)|-1)}{|\Omega|}=\frac{|N_{A}(u)|(|N_{A}(u)|-1)}{|A|(|A|-1)}. 
\end{equation*}
From the construction, it follows that $X=\underset{u\in B}{\sum}X_{u}$. Using the linearity of expectation, we have 
\begin{equation*}
\Ex[X]=\underset{u\in B}{\sum}\Ex[X_{u}]=\underset{u\in B}{\sum}\Pr (X_{u}=1)=\underset{u\in B}{\sum} \frac{|N_{A}(u)|(|N_{A}(u)|-1)}{|A|(|A|-1)}.   
\end{equation*}
We use Jensen's inequality in the form 
\begin{equation*}
\left(\frac{\underset{u\in B}{\sum}|N_{A}(u)|-|B|}{|B|}\right)^{2}=\left(\frac{\underset{u\in B}{\sum}(|N_{A}(u)|-1)}{|B|}\right)^{2}\leq \frac{\underset{u\in B}{\sum}(|N_{A}(u)|-1)^{2}}{\underset{u\in B}{\sum}1},   
\end{equation*}
and further estimate $\Ex[X]$.  
\begin{align*}
\Ex[X]=\underset{u\in B}{\sum} \frac{|N_{A}(u)|(|N_{A}(u)|-1)}{|A|(|A|-1)}
=&\underset{u\in B}{\sum}\frac{(|N_{A}(u)|-1)^{2}}{|A|(|A|-1)}+ \underset{u\in B}{\sum}\frac{(|N_{A}(u)|-1)}{|A|(|A|-1)}\\
         = &\frac{|B|}{|A|(|A|-1)} \frac{\underset{u\in B}{\sum}(|N_{A}(u)|-1)^{2}}{\underset{u\in B}{\sum}1}+ \frac{\underset{u\in B}{\sum}|N_{A}(u)|}{|A|(|A|-1)}-\frac{|B|}{|A|(|A|-1)}\\
         \geq  &\frac{|B|}{|A|(|A|-1)} \left(\frac{\underset{u\in B}{\sum}|N_{A}(u)|-|B|}{|B|}\right)^{2}+\frac{\underset{u\in B}{\sum}|N_{A}(u)|}{|A|(|A|-1)}-\frac{|B|}{|A|(|A|-1)}\\
\end{align*}
Using the \emph{duality formula} $\underset{w\in A}{\sum}|N_{B}(w)|=\underset{u\in B}{\sum}|N_{A}(u)|$ and the condition $|N_{A}(u)|\geq d$ for each $u\in A$, we have the following:-
\begin{align*}
\Ex[X] = &\frac{|B|}{|A|(|A|-1)} \left(\frac{\underset{w\in A}{\sum}|N_{B}(w)|-|B|}{|B|}\right)^{2}+\frac{\underset{w\in A}{\sum}|N_{B}(w)|}{|A|(|A|-1)}-\frac{|B|}{|A|(|A|-1)}\\
         \geq & \frac{|B|}{|A|(|A|-1)} \left(\frac{\underset{w\in A}{\sum}d-|B|}{|B|}\right)^{2}+\frac{\underset{w\in A}{\sum}d}{|A|(|A|-1)}-\frac{|B|}{|A|(|A|-1)}\\
         =& \frac{|B|}{|A|(|A|-1)} \left(\frac{d|A|}{|B|}-1\right)^{2}+\frac{d|A|}{|A|(|A|-1)}-\frac{|B|}{|A|(|A|-1)}\\
=&\frac{|B|}{|A|(|A|-1)} \left(\frac{d^{2}|A|^{2}}{|B|^{2}}-2d\frac{|A|}{|B|}+1\right)+\frac{d}{|A|-1}-\frac{|B|}{|A|(|A|-1)}\\
=&\frac{d^{2}|A|}{|B|(|A|-1)}-\frac{2d}{|A|-1}+\frac{|B|}{|A|(|A|-1)}+\frac{d}{|A|-1}-\frac{|B|}{|A|(|A|-1)}\\
=&\frac{d^{2}|A|}{|B|(|A|-1)}-\frac{d}{|A|-1}>\frac{d^{2}}{|B|}-\frac{d}{|A|-1}>\frac{d^{2}}{|B|}-\frac{d}{|A|}.
\end{align*} 
This implies that there exists at least one $(x,y)\in \Omega$ such that $X(x,y)> \frac{d^{2}}{|B|}-\frac{d}{|A|}$. This completes the proof.
\end{proof}

\begin{remark}
The definition of the random variable $X_{u}:\Omega\rightarrow\mathbb{R}$, in the above proof, provides the formal justification for the term \emph{dependent random choice}. For each sample point $(x,y)\in\Omega$, $u$ is not chosen independently; rather, its selection strictly depends on the pair $(x,y)$.     
\end{remark}

The Lemma~\ref{intersection lemma} can be further generalised to the case of a general graph $G$ and a subset $A\subset V(G)$. Although the proof idea is exactly similar. We consider $A\subset B= V(G)$ for some graph $G$ where each vertices $a\in A$ satisfies $|N_{G}(a)|\geq d$. The duality formula also works here, i.e.,
\begin{equation*}
 \underset{b\in V(G)}{\sum}|N_{G}(b)\cap A|=\underset{b\in V(G)}{\sum}|N_{A}(b)|=\underset{b\in B}{\sum}|N_{A}(b)|=\underset{a\in A}{\sum}|N_{B}(a)|=\underset{a\in A}{\sum}|N_{G}(a)|.
\end{equation*}
Applying Jensen's inequality along with the duality formula, we get the following corollary.

\begin{corollary}\label{intersection lemma 2}
Let $G$ be a graph with $A\subset B=V(G)$ and $|N_{G}(a)|\geq d$ for each $a\in A$. Then there exist at least two distinct vertices $x,y\in A$ such that 
\begin{equation*}
 |N_{G}(x)\cap N_{G}(y)|> \frac{d^{2}}{|B|}-\frac{d}{|A|}=\frac{d^{2}}{|V(G)|}-\frac{d}{|A|}.   
\end{equation*}
\end{corollary}

\section{The Proof of Theorem~\ref{main theorem}}

We assume $m\geq 7$ is an odd integer and $n\geq 4m$ is sufficiently large. We construct the graph $G$, which is three disjoint unions of copies of $\mathbb{K}_{n+1}$. Then $\overline{G}$ is the complete tripartite graph with each part having size $n+1$. From the construction, it is clear that $G\not\subseteq \mathbb{K}_{2,n}$ and $\overline{G}\not\subseteq W_{m}$. This implies $R(\mathbb{K}_{2,n}, W_{m})>3n+3$.

To establish $R(\mathbb{K}_{2,n}, W_{m})=3n+4$, we assume that a graph $G$ with $3n+4$ many vertices does not contain a copy of $\mathbb{K}_{2,n}$. We show that its complement $\overline{G}$ contains a $W_{m}$. On the contrary, let $\overline{G}$ contain no $W_{m}$, we use the $\mathbb{K}_{2,n}$-freeness of $G$ to obtain a certain upper bound on the minimum degree of $G$. Consequently, this shows that $\Delta(\overline{G})$ is sufficiently large, which helps to prove the existence of a $C_{m}$ inside the neighbourhood of the maximum degree vertex of $\overline{G}$. Along with the maximum degree vertex, we get a $W_{m}$ in $\overline{G}$, which is a contradiction. Before proceeding, we need to prove some preliminary results that are used to prove the main theorem.

\begin{lemma}\label{upper bound Del(barG)}
 $\Delta(\overline{G})\leq 2n+2=2(n+1)$. 
\end{lemma}
\begin{proof}
If not, then let $\Delta(\overline{G})\geq 2n+3$. This implies that there exists a vertex $v\in V(G)$ such that $|N_{\overline{G}}(v)|\geq 2n+3$. Since $G$ is $\mathbb{K}_{2,n}$-free, we have $G[N_{\overline{G}}(v)]$ is $\mathbb{K}_{2,n}$-free. Hence, from Theorem~\ref{Our result}, it follows that the graph $\overline{G}[N_{\overline{G}}(v)]$ contains a copy of $C_{m}$. This implies the graph $\overline{G}$ contains a copy of $C_{m}$. This $C_{m}$, along with the vertex $v$, forms a copy of a wheel $W_{m}$, which is a contradiction. This contradicts our assumption that $\overline{G}$ is $W_{m}-$free. 
\end{proof}

Using~\eqref{nbd complementary dependency}, $\Delta(\overline{G})+\delta(G)=|V(G)|-1=3(n+1)$. Thus  using Lemma~\ref{upper bound Del(barG)}, it follows that $\delta(G)\geq (n+1)$. To establish the upper bound of $\delta(G)$, we use the first moment method.

\begin{lemma}\label{upper bound del(G)}
$\delta(G)< \sqrt{3}(n+1)$.
\end{lemma}
\begin{proof}
If possible, let $\delta(G)\geq \sqrt{3}(n+1)$. We show using the first moment method that there exist $x,y\in V(G)$ such that $|N_{G}(x)\cap N_{G}(y)|\geq n$. In contrast to Lemma~\ref{intersection lemma}, the assumptions in the present setting differ, thereby warranting a reformulation of the argument; nonetheless, the underlying principle remains identical. We perform the same random experiment of choosing two vertices one-by-one without replacement. This yields the sample space 
\begin{equation*}
\Omega= \{(x,y): x,y\in V(G), x\neq y\}.
\end{equation*}
From the construction, it follows that $|\Omega|=|V(G)|(|V(G)|-1)=(3n+4)(3n+3)$. Here we consider the probability space $(\Omega, 2^{\Omega}, \Pr)$, where $\Pr(\{\omega\})=\frac{1}{|\Omega|}$ for each $\omega\in\Omega$. We construct a random variable $X: \Omega \rightarrow \mathbb{R}$, where for each $(x,y)\in\Omega$,
\begin{equation*}
X(x,y)= |N_{G}(x)\cap N_{G}(y)|.
\end{equation*}
For a vertex $v\in V$, we construct auxiliary random variable $X_{v}:\Omega\rightarrow\mathbb{R}$, where  for each $(x,y)\in\Omega$,
\begin{align*}
X_{v}(x,y)&=\left\{\begin{array}{lll}
			1& \textnormal{ if  }  v\in N_{G}(x)\cap N_{G}(y)\\
			0& \textnormal{ else }&  
		\end{array}\right..
\end{align*}
From the construction, it follows that $X=\underset{v\in V}{\sum}X_{v}$. Here, we compute the distribution of $X_{v}$, for each $v\in V$,
\begin{equation*}
\Pr(X_{v}=1)=\frac{|N_{G}(v)|(|N_{G}(v)|-1)}{|V(G)|(|V(G)|-1)}\geq \frac{\sqrt{3}(n+1)(\sqrt{3}(n+1)-1)}{(3n+4)(3n+3)}\geq \frac{n}{3n+4}.
\end{equation*}
Hence, using the linearity of expectation, we have

\begin{equation*}
\Ex[X]=\underset{v\in V}{\sum}\Ex[X_{v}]=\underset{v\in V}{\sum}\Pr(X_{v}=1)\geq \underset{v\in V}{\sum}\frac{n}{3n+4}=n.
\end{equation*}
This implies that there exists at least one pair $(x,y)\in\Omega$ such that $|N_{G}(x)\cap N_{G}(y)|=X(x,y)\geq n$, which contradicts \eqref{K_2n freeness dependency}.
\end{proof}

\begin{construction}
Using Lemma~\ref{upper bound Del(barG)} and Lemma~\ref{upper bound del(G)}, we conclude that 
\begin{equation*}
(3-\sqrt{3})(n+1)<\Delta(\overline{G})= (|V(G)|-1)-\delta(G)\leq2(n+1). 
\end{equation*}
We choose a vertex $v\in V(G)$ such that $|N_{\overline{G}}(v)|=\Delta(\overline{G})$ and construct two induced subgraphs namely, $\overline{H}=\overline{G}[N_{\overline{G}}(v)]$ and $H=G[N_{\overline{G}}(v)]$  of the complementary graph $\overline{G}$ and graph $G$ respectively. Thus vertex set of both graphs are same, namely $N_{\overline{G}}(v)$ and its size,
\begin{equation*}
 |N_{\overline{G}}(v)|=|V(\overline{H})|=|V(H)|\in[(3-\sqrt{3})(n+1) , 2(n+1)]\cap\mathbb{Z}.   
\end{equation*}    
\end{construction}

\begin{lemma}\label{maximum bipartition}
Let $\Delta(\overline{G})= \frac{3(n+1)}{2}+t$, where $t\in[-\frac{(2\sqrt{3}-3)(n+1)}{2},\frac{n+1}{2}]$, and for $a\in V(\overline{G})$, $S_{a}\subseteq N_{\overline{G}}(a)$ be a non-empty set such that the induced graph  $\overline{G}[S_{a}]$ forms a bipartite subgraph of $\overline{G}$.  If $A_{a}\sqcup B_{a}=S_{a}$ denotes the bipartition, where $|A_{a}|\geq |B_{a}|$, then $|A_{a}|<\frac{3(n+1)}{4}+\frac{3t}{2}+\frac{3}{4}$.
\end{lemma}
\begin{proof}
If possible for such $a$, let $|A_{a}|= \frac{3(n+1)}{4}+\varepsilon$. Using \eqref{nbd dependency}, for each $w\in N_{\overline{G}}(a)$, we have the equality  $N_{\overline{G}}(w)=N_{\overline{G}[V(\overline{G})-\{a\}]}(w)\sqcup\{a\}$, that leads to
\begin{equation*}
\Delta(\overline{G})\geq|N_{\overline{G}}(w)|=|N_{\overline{G}[V(\overline{G})-\{a\}]}(w)|+1.   
\end{equation*}
Here, we consider the bipartition of 
$V(\overline{G})-\{a\}= A_{a}\sqcup C$, where $C=V(\overline{G})-(A_{a}\sqcup\{a\})$. Then $|C|=3(n+1)+1-|A_{a}|-1=\frac{9(n+1)}{4}-\varepsilon$. In particular, for each $w\in A_{a}\subset S_{a}\subset N_{\overline{G}}(a)$, 
\begin{equation*}
|N_{\overline{G}[V(\overline{G})-\{a\}]}(w)|\leq \Delta(\overline{G})-1= \frac{3(n+1)}{2}+t-1.    
\end{equation*}
Using \eqref{nbd complementary dependency}, we have
\begin{equation*}
|N_{G[V(\overline{G})-\{a\}]}(w)|\geq ((|V(\overline{G})|-1)-1)-\frac{3(n+1)}{2}-t+1=\frac{3(n+1)}{2}-t.  \end{equation*}
Since $G[A_{a}]$ is a complete subgraph of $G$, $|N_{G}(w)\cap A_{a}|=|A_{a}|-1$ and hence 
\begin{equation*}
|N_{G}(w)\cap C|\geq \frac{3(n+1)}{2}-t-(|A_{a}|-1)=\frac{3(n+1)}{4}-t-\varepsilon+1,
\end{equation*}
for each $w\in A_{a}$. We apply Lemma~\ref{intersection lemma} on the partition $A_{a}\sqcup C$ with $d=\frac{3(n+1)}{4}-t-\varepsilon+1$, and conclude that there exist two vertices $x,y\in A_{a}$ such that 
\begin{equation*}
|N_{G}(x)\cap N_{G}(y)\cap C|>\frac{d^{2}}{|C|}-\frac{d}{|A_{a}|}\geq \frac{n+1}{4}-\varepsilon
\end{equation*}

\noindent After plugging the values of $d$, $|A_{a}|$ and $|C|$ in the above expression and further computation shows that 
\begin{equation*}
\frac{d^{2}}{|C|}-\frac{d}{|A_{a}|}= \bigg(\frac{n+1}{4}-\varepsilon\bigg)+\frac{P(n,t,\varepsilon)}{Q(n,t,\varepsilon)},
\end{equation*}
where 
\begin{align*}
 P(n,t,\varepsilon)&=3(4\varepsilon-6t-3)(n+1)^2
+4[3(1-t-\varepsilon)^{2}+(6\varepsilon-9)(1-t-\varepsilon)+3\varepsilon+7\varepsilon^{2}](n+1)+\\
&\hspace{2cm}+4^{2}\varepsilon(1-t-\varepsilon)^{2}+4^{2}\varepsilon(1-t-\varepsilon)-4^{2}\varepsilon^{3}\\
&=3(4\varepsilon-6t-3)(n+1)^2+(3t^{2}+4\varepsilon^{2}+3t+12\varepsilon-6)(n+1)+16\varepsilon(2+t^{2}-3t-3\varepsilon+2t\varepsilon)\\
 Q(n,t,\varepsilon)&=\big(9(n+1)-4\varepsilon\big)\big(3(n+1)+4\varepsilon\big).
\end{align*}
If $\varepsilon\in\left[\frac{3t}{2}+\frac{3}{4},\frac{9(n+1)}{4}\right)$, then for each integer $n\geq1$, $\frac{P(n,t,\varepsilon)}{Q(n,t,\varepsilon)}\geq 0$. Consequently, we get from the above inequality:
\begin{equation*}
 |N_{G}(x)\cap N_{G}(y)\cap C|> \frac{(n+1)}{4}-\varepsilon.   
\end{equation*}
Since these two vertices $x$ and $y$ are adjacent to all the vertices in $A_{a}$ in the graph $G$, satisfies $|N_{G}(x)\cap N_{G}(y)\cap A_{a}|=(|A_{a}|-2)$. Thus we get $|N_{G}(x)\cap N_{G}(y)|\geq n$, which contradicts \eqref{K_2n freeness dependency}. Thus, $\mathbb{K}_{2,n}-$freeness property of $G$ implies $\varepsilon<\frac{3t}{2}+\frac{3}{4}$.
\end{proof}

Depending on the size of $V|(\overline{H})|$, we divide the proof into two cases. In the first case, where
\begin{equation*}
 |V(\overline{H})|\in\left[(3-\sqrt{3})(n+1), \frac{3(n+1)}{2}-1\right),   
\end{equation*}
we employ a probabilistic approach, using Lemma~\ref{intersection lemma}, to show that $\overline{G}[Q]$ forms a non-bipartite graph. In the second case, where 
\begin{equation*}
|V(\overline{H})|\in\left[\frac{3(n+1)}{2}-1, 2(n+1)\right], 
\end{equation*}
some structural rigidity phenomenon arises. These are handled in Lemma~\ref{nbd conservation 2}, together with an application of Lemma~\ref{intersection lemma}.

\subsection{The size of \texorpdfstring{$V(\overline{H})$}{$V(\overline{H})$} lies within the range \texorpdfstring{$[(3-\sqrt{3})(n+1), \frac{3(n+1)}{2}-1)\cap\mathbb{Z}$}{$[(3-\sqrt{3})(n+1), \frac{3(n+1)}{2}-1)\cap\mathbb{Z}$}} 

\paragraph{ Before proceeding, we prove the following lemma, where we show that $V(\overline{H})$ can be decomposed into two subsets, namely the null and dense sets. In particular, we show that the size of the null set is at most $1$.}
\begin{lemma}\label{almost one-tenth}
There exists at most one vertex $x\in V(\overline{H})$, such that $|N_{\overline{H}}(x)|<\frac{|V(\overline{H})|}{10}+1$, i.e. 
\begin{equation*}
 \left|\left\{x\in  V(\overline{H}): |N_{\overline{H}}(x)|<1+\frac{|V(\overline{H})|}{10}\right\}\right|\leq1.   
\end{equation*}
Equivalently, using~\eqref{nbd complementary dependency}, there exists at most one vertex $x\in V(H)$ such that $|N_{H}(x)|\geq \frac{9}{10}|V(H)|-1$.
\end{lemma}
\begin{proof}
If not, then let there exist two vertices $x$ and $y\in V(H)$ such that $|N_{H}(x)|\geq\frac{9}{10}|V(H)|-1$ and $|N_{H}(y)|\geq\frac{9}{10}|V(H)|-1$. Since $G$ is $\mathbb{K}_{2,n}$ free, we have $H$ is $\mathbb{K}_{2,n}$ free. Consequently, \eqref{K_2n freeness dependency} holds for the induced graph $H$ and 
\begin{equation*}
\frac{9}{5}|V(H)|-2-(n-1)\leq|N_{H}(x)|+|N_{H}(y)|-|N_{H}(x)\cap N_{H}(y)|=|N_{H}(x)\cup N_{H}(y)|\leq|V(H)|  
\end{equation*}
From the above inequality, it follows that $(n+1)\geq\frac{4}{5}|V(H)|$. Since 
\begin{equation*}
|N_{\overline{G}}(v)|=|V(H)|=|V(\overline{H})|=\Delta(\overline{G})\geq(3-\sqrt{3})(n+1),    
\end{equation*}
we have $(n+1)\geq\frac{4(3-\sqrt{3)}}{5}(n+1)>(n+1)$, which is a contradiction.
\end{proof}

\begin{remark}
The choice of the quantities $\frac{|V(\overline{H})|}{10}+1$ and the fraction $\frac{1}{10}$ is essential, as this specific selection ultimately leads to a contradiction. In particular, under this choice, we are forced into the incorrect inequality 
\begin{equation*}
(n+1)\geq\frac{4(3-\sqrt{3)}}{5}(n+1)>(n+1).
\end{equation*}
This contradiction arises precisely because of the assumption involving the fraction $\frac{1}{10}$ and therefore justifies its use in our argument.
\end{remark}

We name the set $P=\left\{x\in  V(\overline{H}): |N_{\overline{H}}(x)|<\frac{|V(\overline{H})|}{10}+1\right\}$ constructed above as null set and its complement $Q=V(\overline{H})-P$ as dense set. From the above lemma, we have an immediate corollary.

\begin{corollary}\label{decomposition lemma}
The vertex set $V(\overline{H})=V(H)=N_{\overline{G}}(v)$ admits a decomposition $P\sqcup Q$, with $|P|\leq1$ and for each $x\in Q$, $|N_{\overline{H}}(x)|\geq\frac{|V(\overline{H})|}{10}+1$.
\end{corollary}

From this point onward, to conclude the argument for this range of $|V(\overline{H})|$, our aim is to establish the existence of an odd cycle $C_{m}$ inside $\overline{G}[Q]$. Since $m$ is an odd integer, one of the sufficient conditions of Lemma~\ref{cycle lemma 1} requires us to show that $\overline{G}[Q]$ forms a non-bipartite graph.  

\begin{lemma}\label{nbd conservation}
Within this range, along with the vertex set decomposition $P\sqcup Q=V(\overline{H})$ with $|P|\leq 1$, the induced subgraph $\overline{G}[Q]$ forms a non-bipartite subgraph of $\overline{G}$. 
\end{lemma}
\begin{proof}
 Let $\Delta(\overline{G})=|V(\overline{H})|=\frac{3(n+1)}{2}-\varepsilon$, where $\varepsilon\in\left(1,\frac{(2\sqrt{3}-3)(n+1)}{2}\right]$. In practice, since $|V(\overline{H})|$ is a positive integer, we have  $\varepsilon\geq\frac{3}{2}$.  
If possible  $\overline{G}[Q]$ forms a bipartite graph with bipartition $A\sqcup B$ with $|A|\geq |B|$, then 
\begin{equation*}
|A|\geq \frac{|Q|}{2}=\frac{|V(\overline{H})|-|P|}{2}\geq\frac{3(n+1)}{4}-\frac{\varepsilon+1}{2}.    
\end{equation*}
Using Lemma~\ref{maximum bipartition}, we have $|A|<\frac{3(n+1)}{4}- \frac{3\varepsilon}{2}+\frac{3}{4}$. Thus for each $\varepsilon\in\left[\frac{3}{2},\frac{(2\sqrt{3}-3)(n+1)}{2}\right]$, we have  $\varepsilon\geq\frac{3}{2}>\frac{5}{4}$ and consequently,
\begin{equation*}
 |A|<\frac{3(n+1)}{4}- \frac{3\varepsilon}{2}+\frac{3}{4}=\frac{3(n+1)}{4}- \frac{6\varepsilon-3}{4}< \frac{3(n+1)}{4}-\frac{\varepsilon+1}{2}\leq |A|
\end{equation*}
i.e. $|A<|A|$,  which is a contradiction.
\end{proof}

Recall that Corollary~\ref{decomposition lemma} provides us with a decomposition $P\sqcup Q$ of $V(\overline{H})$ where $|P|\in\{0,1\}$. For both cases, we have shown that $\overline{G}[Q]$ is non-bipartite. Our plan is to use the \emph{connectivity argument} to conclude the existence of a copy of an odd cycle $C_{m}$ inside $\overline{G}[Q]$. 

\begin{lemma}\label{connectivity-first range}
Within this range, along with the vertex set decomposition $P\sqcup Q=V(\overline{H})$. Then the connectivity parameter of $\overline{G}[Q]$ satisfies $\kappa(\overline{G}[Q])\geq2$.
\end{lemma}
\begin{proof}
If not, we assume $\kappa(\overline{G}[Q])\in\{0,1\}$ and 
construct	
\begin{equation*}
	R=\begin{cases}
		\emptyset &\text{ if } \kappa(\overline{G}[Q])=0\\
		\{q\}&\text{ if } \kappa(\overline{G}[Q])=1\text{ i.e. else},
	\end{cases}
\end{equation*}	
where, in the second case, $q$ is a \emph{cut-vertex} of $\overline{G}[Q]$. By construction, the induced subgraph  $\overline{G}[Q-R]$ is not connected.  

On $Q-R$, we construct the relation $\sim$ by declaring $x\sim y$ if and only if there is a path in $\overline{G}[Q-R]$ from 
$x$ to $y$. (We allow the  $2-$tuples $(x,x)$ as paths, where $x\in Q-R$, for the reflexive property.) This is an equivalence relation, and its equivalent classes are precisely the maximal connected components of $\overline{G}[Q-R]$. We call these classes $S_{1},\ldots,S_{m}$, where $m\in[2,9]\cap\mathbb{Z}$, so that 
\begin{equation*}
Q=R\sqcup S_{1}\sqcup\ldots\sqcup S_{m}.
\end{equation*}
Due to the maximal property, for each  $S,S'\in\{S_{1},\ldots,S_{m}\}$ with $S\neq S'$, there is no edge of the form $\{s,s'\}$, where $s\in S$ and $s'\in S'$. Due to these properties, the neighbourhood relation \eqref{nbd dependency} gets a revision, that is, for each $x\in S$ , where $S\in\{S_{1},\ldots,S_{m}\}$
\begin{equation*}
 N_{\overline{G}[S]}(x)\subset N_{\overline{G}[Q]}(x)\subset N_{\overline{G}[S]}(x)\sqcup R.    
\end{equation*}
Since each $S\in\{S_{1},\ldots,S_{m}\}$ forms a connected component, thus we have 
\begin{equation*}
 \frac{V(\overline{H})}{10}+1\leq\delta(\overline{G}[Q])\leq\delta(\overline{G}[S])\leq|S|-1.   
\end{equation*}
This means $|S|\geq2+\frac{V(\overline{H})}{10}$ for each $S\in\{S_{1},\ldots,S_{m}\}$. Since 
\begin{equation*}
|V(\overline{H})|-|P|-|R|=|Q|-|R|=|S_{1}|+\ldots+|S_{m}|\geq2m+m\frac{V(\overline{H})}{10}    
\end{equation*}
holds we have $m\leq9$. 

We choose $B\in\{S_{1},\ldots,S_{m}\}$ such that $|B|=\min\{|S_{1}|,\ldots,|S_{m}|\}$ and construct the set 
\begin{equation*}
 A=(S_{1}\sqcup\ldots\sqcup S_{m})-B.   
\end{equation*}

\noindent{\textsl{Claim} :}  $|A|\in\left[\frac{|V(\overline{H})|}{2}-1,n-1\right]\cap\mathbb{Z}$.
\begin{proof}[\tt{Proof of the claim} :]\renewcommand{\qedsymbol}{}
Here $|A|\geq|B|$. Consequently, $|A|\geq\frac{|Q|-|R|}{2}=\frac{|V(\overline{H})|-|P|-|R|}{2}\geq|B|$. Thus using $|P|\leq1$ and $|R|\leq1$, we have $|A|\geq\frac{|V(\overline{H})|}{2}-1$. We also have $|A|\leq n-1$. If not, let $|A|\geq n$. In the complementary graph of $\overline{G}$, i.e. in $G$, for each $a\in A$ and $b\in B$, $\{a,b\}$ forms an edge in $G$. We choose two vertices $b,b'\in B$, then $N_{G}(b)\cap A=A=N_{G}(b')\cap A$. This implies 
\begin{equation*}
 |N_{G}(b)\cap N_{G}(b')|\geq|N_{G}(b)\cap N_{G}(b')\cap A|=|A|\geq n,   
\end{equation*}
which contradicts \eqref{K_2n freeness dependency} and establishes the claim. 
\end{proof}

We  consider the vertex partition $A\sqcup C\sqcup\{v\}=V(G)$ and construct the induced subgraphs 
\begin{equation*}
 G[A\sqcup C],\quad\overline{G}[A\sqcup C]\quad\text{ and }\quad G[C].   
\end{equation*}
Here $B\subset C$.  For a vertex $w\in B$, since $v$ is the maximum degree vertex in $\overline{G}$ and $w$ is adjacent to $v$ in $\overline{G}$, using \eqref{nbd dependency}, we have 
\begin{equation*}
 N_{\overline{G}}(w)=N_{\overline{G}[A\sqcup C]}(w)\sqcup\{v\}.   
\end{equation*}
This implies, for each $w\in B$,
\begin{equation*}
|N_{\overline{G}[A\sqcup C]}(w)|\leq\Delta(\overline{G})-1=|V(\overline{H})|-1.
\end{equation*}
Using \eqref{nbd complementary dependency} and Lemma~\ref{upper bound Del(barG)}, we have the following:- 
\begin{align*}
|N_{G[A\sqcup C]}(w)|&=(|A|+|C|)-1-|N_{\overline{G}[A\sqcup C]}(w)|\\
&=(|V(G)|-1)-1-|N_{\overline{G}[A\sqcup C]}(w)|\\
&\geq |V(G)|-|V(\overline{H})|-1.
\end{align*}
Revising \eqref{nbd dependency}, we have $N_{G[A\sqcup C]}(b)=N_{G[A]}(b)\sqcup N_{G[C]}(b)$ for each $b\in B$. Therefore for each $b\in B$,
\begin{align*}
|N_{G[C]}(b)|=|N_{G[A\sqcup C]}(b)|-|N_{G[A]}(b)|\geq &|N_{G[A\sqcup C]}(b)|-|A|\\
&\geq |V(G)|-|V(\overline{H})|-|A|-1=|C|-|V(\overline{H})|.   
\end{align*}
Using Corollary~\ref{intersection lemma 2} over the graph $G[C]$, with $d=|C|-|V(\overline{H})|$ and $B\subset C$,  we conclude that there exist two distinct $b,b'\in B$ such that
\begin{equation*}
|N_{G[C]}(b)\cap N_{G[C]}(b')|>\frac{d^{2}}{|C|}-\frac{d}{|B|}\geq n-1-|A|+\frac{P(n,\varepsilon)}{Q(n,\varepsilon)}\geq n-1-|A|,    
\end{equation*}
where 
\begin{align*}
P(n,\varepsilon)&=(15+12\varepsilon)(n+1)^{2}+(16\varepsilon^{2}-36\varepsilon+32)(n+1)-4(4\varepsilon^{3}+\varepsilon^{2}-4)\\
Q(n,\varepsilon)&=[9(n+1)+(4+2\varepsilon)][3(n+1)+(4-2\varepsilon)]. 
\end{align*}
For each $n\geq 1$ and $\varepsilon\in\left(1,\frac{(2\sqrt{3}-3)(n+1)}{2}\right]$, $\frac{P(n,\varepsilon)}{Q(n,\varepsilon)}\geq0$, concludes the last inequality. While estimating the quantity $\frac{d^{2}}{|C|}-\frac{d}{|B|}$, we have used the following inequalities:- 
\begin{align*}
\frac{|V(\overline{H})|^{2}}{|C|}\geq&\frac{|V(\overline{H})|^{2}}{|V(G)|-\frac{|V(\overline{H})|}{2}}=\frac{[3(n+1)-2\varepsilon]^{2}}{9(n+1)+(4+2\varepsilon)},\\
\frac{|V(\overline{H})|-|C|}{|B|}\geq&\frac{\frac{3}{2}|V(\overline{H})|-|V(G)|}{\frac{|V(\overline{H})|}{2}+1}=3-\frac{|V(G)|+3}{\frac{|V(\overline{H})|}{2}+1}=3-\frac{12(n+1)+16}{3(n+1)+(4-2\varepsilon)}.
\end{align*}
Here $N_{G}(b)\cap N_{G}(b')\cap A=A$, $v\notin N_{G}(b)\cap N_{G}(b')$, $C\subset V(G)$ and $G[C]$ forms a subgraph of $G$, these imply
for two distinct vertices $b,b'\in B$,
\begin{align*}
|N_{G}(b)\cap N_{G}(b')|=&|N_{G}(b)\cap N_{G}(b')\cap A|+|N_{G}(b)\cap N_{G}(b')\cap C|\\
=&|N_{G}(b)\cap N_{G}(b')\cap A|+|N_{G[C]}(b)\cap N_{G[C]}(b')|\\
&\geq |A|+(n-1)-|A|+1=n.
\end{align*}
This contradicts \eqref{K_2n freeness dependency}.
\end{proof}

\begin{construction}\label{construction}
Here we construct the choice of an integer $n\geq1$ and an odd integer $m\geq1$ with $n\geq4m$. Our straightforward intuition is that one of its connected components of the subgraph $\overline{G}[Q]$ contains an odd cycle. Thus, our motive is to use  Lemma~\ref{cycle lemma 1}.  Jotting together Corollary~\ref{decomposition lemma} and  Lemma~\ref{nbd conservation}, we have $V(\overline{H})=P\sqcup Q$, with $|P|\leq1$ and the induced subgraph  $\overline{G}[Q]$ forms a non-bipartite graph with vertex set $Q$,
\begin{equation}\label{sizeQ}
|Q|= |V(\overline{H})|-|P|\geq|V(\overline{H})|-1
\end{equation}
Using the neighbourhood relation \eqref{nbd dependency} we have $\delta(\overline{G}[Q])\geq\delta(\overline{H})-|P|\geq\delta(\overline{H})-1$. Using the property of $Q$, we also have $\delta(\overline{G}[Q])\geq\frac{|V(\overline{H})|}{10}$. Hence 
\begin{equation}\label{min deg Q}
\delta(\overline{G}[Q])\geq\max\left\{\delta(\overline{H})-1,\frac{|V(\overline{H})|}{10}\right\}\geq\frac{|V(\overline{H})|}{10}\geq\frac{|Q|}{10}.
\end{equation}
The inequality \eqref{min deg Q} motivates choosing the preassigned number $\varepsilon$ to be equal to $\frac{1}{10}$. In addition to \eqref{min deg Q}, if the induced subgraph $\overline{G}[Q]$ satisfies the connectivity condition $\kappa(\overline{G}[Q])\geq2$, then Lemma~\ref{cycle lemma 1} becomes applicable. In that case, together with \eqref{sizeQ} and  \eqref{min deg Q}, plugging $r=\frac{|Q|}{10}$, in Lemma~\ref{cycle lemma 2} we conclude
\begin{equation}\label{odd circumferrence cicle}
\mathop{oc}(\overline{G}[Q])> \frac{2(|V(\overline{H})|-1)}{10}-1=\frac{|V(\overline{H})|}{5}-\frac{4}{5}.
\end{equation} 
The above bound \eqref{odd circumferrence cicle} motivates us to choose $m$ accordingly. Using Lemma~\ref{cycle lemma 2} we choose a positive real number $K=K(\frac{1}{10})$. Depending on such $K$, we construct the positive real number $45K10^{4}$. Primarily, we choose the least integer $N\geq1$ satisfying
\begin{equation}\label{large integer construction}
(3-\sqrt{3})(N+1)\geq 45K10^{4}.
\end{equation}
For each graph $\Gamma$ with $n$ vertices, where $n\geq N$  and  $\delta(\Gamma)>\frac{n}{10}$, contains an odd cycle $C_{m}$ for each odd integer
\begin{equation*}
m\in[K, oc(\Gamma)-K].
\end{equation*}
Secondarily, we choose such an odd integer $m$. Motivating from \eqref{odd circumferrence cicle}, we have
\begin{equation*}
A(n)=\frac{3-\sqrt{3}}{5}(n+1)-1=\frac{(3-\sqrt{3})(n+1)-1}{5}-\frac{4}{5}\leq\frac{|V(\overline{H})|}{5}-\frac{4}{5}.    
\end{equation*}
Since $K\leq \frac{5}{7-4\sqrt{3}}K+\frac{3+\sqrt3}{7-4\sqrt{3}}=B(K)$ holds, we have 
\begin{equation*}
\left[B(K),A(n)-K\right]\subset\left[K,A(n)-K\right]\subset\left[K, \mathop{oc(\overline{G}[Q])}-K\right]
\end{equation*}
Here we observe the following:-

\noindent{\textsl{Claim} :}  If $n\geq 4m$, then $m\geq B(K)$ if and only if $m\leq A(n)-K$. 
\begin{proof}[\tt{Proof of the claim} :]\renewcommand{\qedsymbol}{}
Using $n\geq 4m$, we estimate,
\begin{align*}
A(n)-K=\left(\frac{3-\sqrt{3}}{5}\right)n+\frac{3-\sqrt{3}}{5}-1-K\geq&\left(\frac{3-\sqrt{3}}{5}\right)4m+\frac{3-\sqrt{3}}{5}-1-K\notag\\
&=m+\left(\frac{7-4\sqrt{3}}{5}\right)m-\frac{2+\sqrt{3}}{5}-K
\end{align*}
Then 
$A(n)-K\geq m+\left(\frac{7-4\sqrt{3}}{5}\right)m-\frac{2+\sqrt{3}}{5}-K\geq m$ if and only if 
\begin{align*}
\left(\frac{7-4\sqrt{3}}{5}\right)m-\frac{2+\sqrt{3}}{5}-K\geq0\Leftrightarrow \frac{7-4\sqrt{3}}{5}m\geq K+\frac{2+\sqrt{3}}{5}\Leftrightarrow m\geq B(K).
\end{align*}
Hence, the claim is established. 
\end{proof}	
\noindent Using the above claim, if we choose $m\geq B(K)$ and $n\geq 4m$, then $m\in\left[B(K),A(n)-K\right]$. Consequently, $m\in\left[K, \mathop{oc(\overline{G}[Q])}-K\right]$. This concludes the construction of the parameters $n$ and $m$.
\end{construction}

 Therefore the non-bipartite graph $\overline{G}[Q]$, satisfies $\kappa(\overline{G}[Q])\geq 2$ and $\delta(\overline{G}[Q])\geq \frac{|Q|}{10}$. Using Lemma~\ref{cycle lemma 1} for $r=\frac{|Q|}{10}$ and following \eqref{odd circumferrence cicle}, we conclude the graph $\overline{G}[Q]$ contains an odd cycle of length at least $\frac{|V(\overline{H})|}{5}-\frac{4}{5}$. Following the consequence as described after \eqref{odd circumferrence cicle}, we take 
\begin{equation*}
 m\in\left[K, \mathop{oc(\overline{G}[Q])}-K\right].   
\end{equation*}
Finally, using Lemma~\ref{cycle lemma 2}, we conclude that there exists a copy of an odd cycle $C_{m}$ in $\overline{G}[Q]$. Such $C_{m}$ together with the vertex $v$ forms a copy of a wheel $W_{m}$ in $\overline{G}$, which is a contradiction. This completes the study if $|V(\overline{H})|\in [(3-\sqrt{3})(n+1), \frac{3(n+1)}{2}-1)$.

\begin{remark}
The choice
\begin{equation*}
m\geq\begin{cases}
\frac{5K}{7-4\sqrt{3}}+\frac{2+\sqrt3}{7-4\sqrt{3}}&\text{ if } P=\emptyset\\
\frac{5K}{7-4\sqrt{3}}+\frac{3+\sqrt3}{7-4\sqrt{3}}=B(K)&\text{ if } |P|=1.
\end{cases}    
\end{equation*}
Thus, the choice of $m\geq B(K)$ is sufficient throughout the proof of Theorem~\ref{main theorem}.
\end{remark}

\subsection{The size of \texorpdfstring{$V(\overline{H})$}{$V(\overline{H})$} lies within the  range \texorpdfstring{$\left[\frac{3(n+1)}{2}-1, 2(n+1)\right]\cap\mathbb{Z}$}{$\left[\frac{3(n+1)}{2}-1, 2(n+1)\right]\cap\mathbb{Z}$}}

\paragraph{We first show that $V(\overline{H})$ obtains a similar decomposition as Corollary~\ref{decomposition lemma}. We first prove a general decomposition lemma, which we use later in our proof.}

\begin{lemma}\label{general decomposition lemma}
Within this range, for each $a\in V(\overline{G})$ with $|N_{\overline{G}}(a)|\geq\frac{3(n+1)}{2}-1$, the set $N_{\overline{G}}(a)$   admits a decomposition $P_{a}\sqcup Q_{a}=N_{\overline{G}}(a)$, where $|P_{a}|\leq 1$ and
\begin{equation*}
 P_{a}=\left\{x\in N_{\overline{G}}(a):|N_{\overline{G}}(x)\cap N_{\overline{G}}(a)|<\frac{|N_{\overline{G}}(a)|}{6}-\frac{1}{3}\right\}.   
\end{equation*}
\end{lemma}
\begin{proof}
If possible let for some $a\in V(\overline{G})$ with $|N_{\overline{G}}(a)|\geq\frac{3(n+1)}{2}-1$ satisfy $|P_{a}|\geq2$. This implies there exist two vertices $x$ and $y\in N_{\overline{G}}(a)$ such that for $u\in\{x,y\}$
\begin{equation*}
 |N_{\overline{G}[N_{\overline{G}}(a)]}(u)|=|N_{\overline{G}}(u)\cap N_{\overline{G}}(a)|>\frac{|N_{\overline{G}}(a)|}{6}-\frac{1}{3},  \end{equation*}
and using~\eqref{nbd complementary dependency} for the induced graph $G[N_{\overline{G}}(a)]$ we have, 
\begin{equation*}
 |N_{G[N_{\overline{G}}(a)]}(u)|<\frac{5|N_{\overline{G}}(a)|}{6}-1+\frac{1}{3}.
 \end{equation*}

Since $G$ is $\mathbb{K}_{2,n}$ free, we have the induced graph $G[N_{\overline{G}}(a)]$ is $\mathbb{K}_{2,n}$ free. Consequently, \eqref{K_2n freeness dependency} holds for the induced graph $G[N_{\overline{G}}(a)]$ and therefore 
\begin{align*}
|N_{\overline{G}}(a)|&\leq|\left(N_{G}(x)\cup N_{G}(y)\right)\cap N_{\overline{G}}(a)|=|N_{G[N_{\overline{G}}(a)]}(x)\cup N_{G[N_{\overline{G}}(a)]}(y)|\\
&=|N_{G[N_{\overline{G}}(a)]}(x)|+|N_{G[N_{\overline{G}}(a)]}(y)|-|N_{G[N_{\overline{G}}(a)]}(x)\cap N_{G[N_{\overline{G}}(a)]}(y)|\\
&<\frac{5|N_{\overline{G}}(a)|}{3}-2+\frac{2}{3}-(n-1)\leq\frac{3(n+1)}{2}-1\leq|N_{\overline{G}}(a)|. 
\end{align*}
That is $|N_{\overline{G}}(a)|<|N_{\overline{G}}(a)|$ holds, which yields a contradiction.      
 \end{proof}

\begin{corollary}\label{decomposition lemma 2}
Within this range,  $V(\overline{H})$ admits a decomposition $P\sqcup Q=V(\overline{H})$, with $|P|\leq1$ and 
\begin{equation*}
P=\left\{x\in V(\overline{H}): |N_{\overline{H}}(x)|<\frac{|V(\overline{H})|}{6}-\frac{1}{3}\right\}\\
\end{equation*}
\end{corollary}
\begin{proof}
Using Lemma~\ref{upper bound Del(barG)}, we have the vertex $v$ satisfies $|N_{\overline{G}}(v)|=|V(\overline{H})|\geq\frac{3(n+1)}{2}-1$ for each integer $n\geq1$. Thus Lemma~\ref{general decomposition lemma} holds as a particular case.   
\end{proof}

\noindent In the next lemma, we show that $\overline{G}[Q]$ is non-bipartite for each $|P|\in \{0,1\}$. Then we use that non-bipartite property of $\overline{G}[Q]$ to conclude our proof. With respect to this decomposition, note that 
\begin{equation*}
\delta(\overline{G}[Q])\geq\begin{cases}
\frac{|V(\overline{H})|}{6}-\frac{1}{3}\geq \frac{n+1}{4}-\frac{1}{2}&\text{ if } |P|=0\text{ or }\overline{G}[Q]=\overline{H}\\
\frac{|V(\overline{H})|}{6}-\frac{1}{3}-1\geq \frac{n+1}{4}-\frac{3}{2} &\text{ if } |P|=1\text{ using }\eqref{nbd dependency}.
\end{cases}   
\end{equation*}
Thus, in any case $\delta(\overline{G}[Q])\geq\frac{n+1}{4}-\frac{3}{2}$.

\begin{lemma}\label{nbd conservation 2}
Within this range, along with the vertex set decomposition $P\sqcup Q=V(\overline{H})$, if for each $x\in Q'\subset Q$, where 
\begin{equation*}
Q'=\left\{x\in Q:|N_{\overline{G}}(x)|\geq\frac{3}{2}(n+1)-1\right\},
\end{equation*}
the induced subgraph  $\overline{G}[Q_{x}]$ forms a bipartite graph, then $\overline{G}[Q]$ forms a non-bipartite subgraph of the graph $\overline{G}$.
\end{lemma}
\begin{proof}
From Corollary~\ref{decomposition lemma 2}, we have a decomposition of $V(\overline{H})=P\sqcup Q$, where $|P|\leq 1$. Hence we assume $|Q|=\frac{3(n+1)}{2}+t$ for some $t\in [-2, \frac{n+1}{2}]$. The lower bound of $t$ follows since $|Q|\geq |V(\overline{H})|-1$. If possible, let $\overline{G}[Q]$ form a bipartite graph with bipartition $A\sqcup B=Q$, where $|A|\geq |B|$. Hence, $|A|\geq \frac{3(n+1)}{4}+\frac{t}{2}=\frac{|Q|}{2}\geq|B|$. So let 
\begin{align*}
|A|&=\frac{3(n+1)}{4}+\frac{t}{2}+\rho\textup{ and }|B|=\frac{3(n+1)}{4}+\frac{t}{2}-\rho,
\end{align*}
where  $\rho\in[0,(\frac{n+1}{4}-\frac{t}{2})]$. The upper bound of $\rho$ follows from Claim~1, and the lower bound of $\rho$ follows from the lower bound of $|A|$. 

\noindent{\textsl{Claim~1} :} $|B|\geq\frac{n+1}{2}+t$
\begin{proof}[\tt{Proof of the claim} :]\renewcommand{\qedsymbol}{}
If not, then $|B|<\frac{n+1}{2}+t$. Hence 
\begin{equation*}
|A|=|Q|-|B|>|Q|-\frac{n+1}{2}-t=\frac{3(n+1)}{2}+t-\frac{n+1}{2}-t=(n+1), 
\end{equation*}
i.e. $|A|\geq n+2$. Since $G[A]$ forms a complete graph, it contains a copy of $\mathbb{K}_{2,n}$, which leads to the contradiction.   	
\end{proof}

We construct $C= V(\overline{G})- A=V(G)-A$ and 
\begin{equation*}
A'=\left\{a\in A:|N_{\overline{G}}(a)|\geq\frac{3(n+1)}{2}-1\right\}.    
\end{equation*}

Therefore, $|C|=\frac{9}{4}(n+1)-\frac{t}{2}-\rho+1$, $B\subset C$ and $A'\subset A$.

\begin{figure}[h!]
\centering
\begin{tikzpicture}[every node/.style={font=\small}, >=latex]
\node (Aellipse) at (-3,0)
[draw, very thick, ellipse,fill=blue!15,minimum width=1.25cm, minimum height=3.5cm] {};
\node at (Aellipse.center) [font=\large] {$A$};
\def\Csidelength{3.5cm}
\coordinate (Ccenter) at (2,0);
\coordinate (Csw) at ($(Ccenter)+(-\Csidelength/2,-\Csidelength/2)$); 
\draw[very thick,fill=red!12] (Csw) rectangle ++(\Csidelength,\Csidelength);
\node[below=2cm of Ccenter, yshift=\Csidelength/4, font=\large] (Clabel) {$C$};
\begin{scope}
\coordinate (Bcenter) at ($(Ccenter)+(-1,1)$);
\draw[very thick] 
($(Bcenter)+(-0.45cm,0.0cm)$)
.. controls ($(Bcenter)+(-0.3cm,0.7cm)$) and ($(Bcenter)+(0.6cm,0.6cm)$) ..
($(Bcenter)+(0.8cm,0.1cm)$)
.. controls ($(Bcenter)+(0.6cm,-0.3cm)$) and ($(Bcenter)+(-0.2cm,-0.6cm)$) ..
($(Bcenter)+(-0.6cm,-0.15cm)$)
.. controls ($(Bcenter)+(-0.8cm,0.05cm)$) and ($(Bcenter)+(-0.6cm,0.25cm)$) ..
($(Bcenter)+(-0.45cm,0.0cm)$);
\node at (Bcenter) [font=\large] {$B$};
\end{scope}
\coordinate (vpoint) at (3,0);
\fill (vpoint) circle (3.5pt);
\node[right=3mm of vpoint, font=\large] {$v$};
\end{tikzpicture}
\caption{Visual depiction of decompositions}
\end{figure}
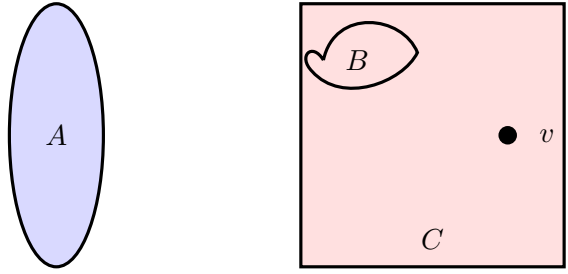 
For each $a\in A'$, we have a decomposition $N_{\overline{G}}(a)=P_{a}\sqcup Q_{a}$ follows from Lemma~\ref{general decomposition lemma}. By our assumption, $\overline{G}[Q_{a}]$ is bipartite for each $a\in A'$. So for each $a,a'\in A'$, we let $A_{a}\sqcup B_{a}$  and $A_{a'}\sqcup B_{a'}$ be the respective bipartition  of the induced subgraph $\overline{G}[Q_{a}]$ and  $\overline{G}[Q_{a'}]$. Without loss of generality, we assume $|A_{a}|\geq\frac{|Q_{a}|}{2}\geq|B_{a}|$ for each $a\in A'$. We  write, for each $a\in A'$,
\begin{align*}
|A_{a}|&= \frac{3(n+1)}{4}-1+k_{a}\approx\frac{|Q_{a}|}{2}+k_{a}\\
|B_{a}|&\geq\frac{3(n+1)}{4}-1-k_{a}\approx\frac{|Q_{a}|}{2}-k_{a},
\end{align*}
where $k_{a}\in [0, \frac{n+1}{4}+1]$. The lower and upper bound of $k_{a}$ holds, since $\frac{|Q_{a}|}{2}\leq|A_{a}|\leq (n+1)$. If possible, let $|B_{a}|<\frac{n+1}{2}-1$, which meansG
\begin{equation*}
|A_{a}|=|Q_{a}|-|B_{a}|>\frac{3(n+1)}{2}-1-\frac{n+1}{2}+1=n+1,  
\end{equation*}
i.e. $|A_{a}|\geq n+2$. Using the assumed bipartite property of $\overline{G}[Q_{a}]$, we have $G[A_{a}]$ forms a complete graph, hence contains a copy of $\mathbb{K}_{2,n}$. This leads to a contradiction. Hence $|B_{a}|\geq\frac{n+1}{2}-1$ and consequently we have 
\begin{align*}
|B_{a}|\geq&\max\left\{\frac{3(n+1)}{4}-1-k_{a},\frac{n+1}{2}-1\right\}.    
\end{align*}

\noindent{\textsl{Claim~2} :} For $a,a'\in A'$, with $a\neq a'$, the following statements hold.
\begin{enumerate}[(a)]
\item If  $|A_{a}\cap B_{a'}|\geq 2$, then $|A_{a}\cup B_{a'}|\leq (n+1)$.
\item If  $|A_{a}\cap A_{a'}|\geq 2$, then $|A_{a}\cup A_{a'}|\leq (n+1)$.
\item If  $|B_{a}\cap B_{a'}|\geq 2$, then $|B_{a}\cup B_{a'}|\leq (n+1)$.
\item If  $|A_{a}\cap B|\geq 2$, then $|A_{a}\cup B|\leq (n+1)$.
\item If  $|B_{a}\cap B|\geq 2$, then $|B_{a}\cup B|\leq (n+1)$.
\end{enumerate}

\begin{proof}[\tt{Proof of the claim} :]\renewcommand{\qedsymbol}{}
Let $|A_{a}\cap B_{a'}|\geq 2$ hold and $x,y\in A_{a}\cap B_{a'}$. If possible, $|A_{a}\cup B_{a'}|\geq (n+2)$ holds, then due to the bipartite property of $\overline{G}[N_{\overline{G}}(a)]$, we get both the vertices  $x$ and $y$  are adjacent to each vertex $u\in (A_{a}\cup B_{a'})- \{x,y\}$ in $G$. This results in a copy of $\mathbb{K}_{2,n}$ in $G$, leading to a contradiction and thus establishing (a). Using analogous arguments, we also obtain (b), (c), (d) and (e).
\end{proof}

\noindent{\textsl{Claim~3} :}  For each $a\in A'$,  exactly one of $A_{a}$ or $B_{a}$ contains $v$. If $v\in S_{a}\in\{A_{a},B_{a}\}$ we call the other by $T_{a}$, i.e $v\notin T_{a}$,  then  we have the following:-
\begin{enumerate}[(a)]
\item $B\cap S_{a}=\emptyset$.
\item $|B\cap T_{a}|\geq\delta(\overline{G}[Q])\geq\frac{n+1}{4}-\frac{3}{2}$.

\item If $S_{a}= B_{a}$ and then for each $n\geq13$,  
\begin{equation*}
|B\cap A_{a} |\geq \max \left\{\frac{n+1}{2}+\frac{t}{2}-1-\rho+k_{a}, \frac{n+1}{4}-\frac{3}{2}\right\}.
\end{equation*}

\item If $S_{a}= A_{a}$, then for each $n\geq13$,  then 
\begin{equation*}
|B\cap B_{a} |\geq \max \left\{\frac{n+1}{2}+\frac{t}{2}-1-\rho-k_{a}, \frac{n+1}{4}-\frac{3}{2}\right\}.
\end{equation*}
\end{enumerate}
\begin{proof}[\tt{Proof of the claim} :]\renewcommand{\qedsymbol}{}
If not, let $x\in B\cap S_{a}$. Then $x$ is adjacent to $v$ in $\overline{G}$; that is,  
$\{x,v\}\subset S_{a}$ and $\{x,v\}\in E(\overline{G})$.  This contradicts the assumption that each edge $\{x,v\}$ cannot lie entirely within one part $S_{a}$ of the assumed induced bipartite subgraph $\overline{G}[N_{\overline{G}}(a)]$.  Thus $B\cap S_{a}=\emptyset$, i.e (a) holds.
	
For the other
\begin{align*}
N_{\overline{G}[Q]}(a)=N_{\overline{G}}(a)\cap Q&=(N_{\overline{G}}(a)\cap A)\sqcup (N_{\overline{G}}(a)\cap B)\\
	    &=\emptyset\sqcup (A_{a}\sqcup B_{a})\cap B\\
		&=(A_{a}\cap B)\sqcup (B_{a}\cap B)=(S_{a}\cap B)\sqcup (T_{a}\cap B)
\end{align*}
Therefore $|B\cap T_{a}|\geq\delta(\overline{G}[Q])$. Finally, $\delta(\overline{G}[Q])\geq\frac{n+1}{4}-\frac{3}{2}$ holds, using Corollary~\ref{decomposition lemma 2}.
	
Since $v\in B_{a}$, using (a) and (b), we have  $|B\cap A_{a}|\geq\frac{n+1}{4}-\frac{3}{2}\geq2$ for each $n\geq13$. By the Construction~\ref{construction}, $n$ is sufficiently large. Using (d) of Claim~2, $|B\cup A_{a}|\leq (n+1)$. Thus we have a revision on the size of $B\cup A_{a}$, namely
\begin{equation*}
|A_{a}|+|B|-|A_{a}\cap B|=|A_{a}\cup B|\leq(n+1),   
\end{equation*}
Using the size $|B|$ and putting $|A_{a}|\geq  \frac{3(n+1)}{4}-1+k_{a}$ in the above inequality, we conclude that (c) holds. The inequality of (d) holds in a similar way.
\end{proof}

Jotting together Claim~2~(d)~and~(e),  since  $v\in S_{a}$, we have for each integer $n\geq13$,  
\begin{equation*}
|B\cap T_{a} |\geq\begin{cases}
		\max \left\{\frac{n+1}{2}+\frac{t}{2}-1-\rho+k_{a}, \frac{n+1}{4}-\frac{3}{2}\right\}&\text{ if } T_{a}=A_{a}\\
		\max \{\frac{n+1}{2}+\frac{t}{2}-1-\rho-k_{a}, \frac{n+1}{4}-\frac{3}{2}\}&\text{ if } T_{a}=B_{a}
	\end{cases}.
\end{equation*}

\noindent{\textsl{Claim~4} :}  For each $a,a'\in A'$ with $a\neq a'$, we have the following:-
\begin{enumerate}[(a)]
\item $|[N_{\overline{G}}(a)\cup N_{\overline{G}}(a')]\cap C|\geq 2(n+1)+1\geq\Delta(\overline{G})+1$.
\item If $|A_{a}\cap A_{a'}|\geq 2$,  then  $|B_{a}\cap B_{a'}|\leq 1$. 
\item If  $|A_{a}\cap B_{a'}|\geq 2$, then  $|B_{a}\cap A_{a'}|\leq 1$.
\end{enumerate}
\begin{proof}[\tt{Proof of the claim} :]\renewcommand{\qedsymbol}{}
If not, let $|(N_{\overline{G}}(a)\cup N_{\overline{G}}(a'))\cap C|\leq 2(n+1)$, for some $a,a'\in A'$ with $a\neq a'$. Using \eqref{nbd complementary dependency}, we have the (structural) identity:-
\begin{equation*}
	C-[(N_{\overline{G}}(a)\cup N_{\overline{G}}(a'))\cap C]=N_{G}(a)\cap N_{G}(a')\cap C.  
\end{equation*}
Thus 
\begin{equation*}
	|N_{G}(a)\cap N_{G}(a')\cap C |=|C-[N_{\overline{G}}(a)\cup N_{\overline{G}}(a')]|=|C|-|(N_{\overline{G}}(a)\cup N_{\overline{G}}(a'))\cap C|.
\end{equation*}
Also recall that, due to the assumed bipartite structure, $|N_{G}(a)\cap N_{G}(a')\cap A|=|A|-2$. Thus 
\begin{align*}
	|N_{G}(a)\cap N_{G}(a')|=&|N_{G}(a)\cap N_{G}(a')\cap A|+|N_{G}(a)\cap N_{G}(a')\cap C|\\
	=&|A|-2+|C|-|(N_{\overline{G}}(a)\cup N_{\overline{G}}(a'))\cap C|\\
	=&|V(G)|-2-|N_{\overline{G}}(a)\cup N_{\overline{G}}(a'))\cap C|\\
	\geq&3(n+1)+1-2-2(n+1)=n,  
\end{align*}
This contradicts \eqref{K_2n freeness dependency}. Thus (a) holds. 

If (b) does not hold, then there exist $a,a'\in A'$ with $a\neq a'$ such that $|A_{a}\cap A_{a'}|\geq 2$ and  $|B_{a}\cap B_{a'}|\geq 2$. In that case, using Claim~2~(b)~and~(c), we have $|(A_{a}\cup A_{a'})|\leq n+1$ and $|(B_{a}\cup B_{a'})|\leq n+1$. Thus, 
\begin{equation*}
    |(A_{a}\cup A_{a'}\cup B_{a}\cup B_{a'})|\leq |(A_{a}\cup A_{a'})|+|(B_{a}\cup B_{a'})|\leq 2(n+1)
\end{equation*}
contradicting Claim~4~(a). Hence (b) holds, and similarly, (c) also holds.
\end{proof}

\noindent{\textsl{Claim~5} :} For each integer $n\geq54$, there does not exist a $4-$set $U$, where  $U\subset A'$, such that $|T_{a}\cap T_{a'}|\leq1$ for each $a,a'\in U$ with $a\neq a'$.  
\begin{proof}[\tt{Proof of the claim} :]\renewcommand{\qedsymbol}{}
For an integer $n\geq54$, if possible there exists a $4-$set $U\subset A'$ such that $|T_{a}\cap T_{a'}|\leq1$ for each $a,a'\in U$ with $a\neq a'$. This implies $\left|\left(\underset{a\in R}{\cap}T_{a}\right)\cap B\right|\leq1$ for each $R\subset U$ with $|R|\in\{2,3,4\}$.

We decompose the set $U=\{p\}\sqcup(U-\{p\})$ for some $p\in U$ and using the inclusion and exclusion principle,
\begin{align*}
|B|\geq|\underset{a\in U}{\cup}(T_{a}\cap B)|=&\underset{a\in U}{\sum}|T_{a}\cap B|+\underset{\substack{R\subset U\\|R|\geq2}}{\sum}(-1)^{|R|+1}\left|\left(\underset{a\in R}{\cap}T_{a}\right)\cap B\right|\\ 
&\geq\underset{a\in U}{\sum}|T_{a}\cap B|-\binom{4}{2}-\binom{4}{4}>|B|,
\end{align*}
which leads to the contradiction. To see the estimation, we use the formula 
\begin{equation*}
\max\{a,b\}+\max\{c,d\}=\max\{a+c,b+c,a+d,b+d\}\geq a+c    
\end{equation*}
for each $a,b,c,d\in\mathbb{R}$. Since the integer $n\geq54$, we have
\begin{align*}
|T_{p}\cap B|+\underset{a\in U-\{p\}}{\sum}|T_{a}\cap B|\geq&\begin{cases}
		\max \left\{\frac{n+1}{2}+\frac{t}{2}-1-\rho+k_{p}, \frac{n+1}{4}-\frac{3}{2}\right\}&\text{ if } T_{p}=A_{p}\\
		\max \left\{\frac{n+1}{2}+\frac{t}{2}-1-\rho-k_{p}, \frac{n+1}{4}-\frac{3}{2}\right\}&\text{ if } T_{p}=B_{p}
 \end{cases}+\\
 &\hspace{1cm}+\underset{a\in U-\{p\}}{\sum}\begin{cases}
		\max \left\{\frac{n+1}{2}+\frac{t}{2}-1-\rho+k_{a}, \frac{n+1}{4}-\frac{3}{2}\right\}&\text{ if } T_{a}=A_{a}\\
		\max \left\{\frac{n+1}{2}+\frac{t}{2}-1-\rho-k_{a}, \frac{n+1}{4}-\frac{3}{2}\right\}&\text{ if } T_{a}=B_{a}
\end{cases}\\
 \geq&\frac{n+1}{2}+\frac{t}{2}-1-\rho-k_{p}+\underset{a\in U-\{p\}}{\sum}\frac{n+1}{4}-\frac{3}{2}\\
 &=|B|-k_{p}+\frac{n+1}{2}-\frac{11}{2}\\
 &\geq|B|+\frac{n+1}{2}-\left(\frac{n+1}{4}+1\right)-\frac{11}{2}\\
 &\geq|B|+\frac{n+1}{4}-\frac{13}{2}=|B|+\binom{4}{2}+\binom{4}{4}+\frac{n-53}{4}\\
 &>|B|+\binom{4}{2}+\binom{4}{4}.
\end{align*}
Hence, the claim is established.
\end{proof}

\noindent{\textsl{Claim~6} :} For each integer $n\geq19$, there does not exist a $4-$set $U\subset A'$ such that $|T_{a}\cap T_{a'}|\geq2$ for each $a,a'\in U$.
\begin{proof}[\tt{Proof of the claim} :]\renewcommand{\qedsymbol}{}
 Let there exist such a set $U= \{p,q,r,s\}$. Since $|T_{a}\cap T_{a'}|\geq2$, using Claim~4~(b)~and~(c), we conclude that  $|S_{a}\cap S_{a'}|\leq 1$, for each $a,a'\in\{p,q,r,s\}$ with $a\neq a'$. Since in this case $v\in S_{a}\cap S_{a'}$, for each $a,a',a''\in\{p,q,r,s\}$, we conclude 
\begin{equation*}
S_{a}\cap S_{a'}\cap S_{a''}=\{v\}=S_{a}\cap S_{a'}=S_{p}\cap S_{q}\cap S_{r}\cap S_{s} 
\end{equation*}
Finally, using the inclusion and exclusion principle, we have 
\begin{align*}
|C|\geq|(S_{p}\cup S_{q}\cup S_{r}\cup S_{s})\sqcup B|&=|S_{p}|+|S_{q}|+|S_{r}|+|S_{s}|+\underset{\substack{R\subset\{p,q,r,s\}\\|R|\geq2}}{\sum}(-1)^{|R|+1}+|B|\\
&=|S_{p}|+|S_{q}|+|S_{r}|+|S_{s}|-4+1+|B|\\
&\geq |B_{p}|+|B_{q}|+|B_{r}|+|B_{s}|-3+|B|\\
&\geq |B|-3+|B_{p}|+\underset{a\in\{q,r,s\}}{\sum}\left(\frac{n+1}{2}-\frac{3}{2}\right)\\
&\geq\frac{3(n+1)}{4}+\frac{t}{2}-\rho-3+\frac{3(n+1)}{4}-\frac{1}{2}-k_{p}+\frac{3(n+1)}{2}-\frac{9}{2}\\
&=|C|+t-k_{p}+\frac{3(n+1)}{4}-8\\
&\geq|C|-8+\frac{3(n+1)}{4}+t-\left(\frac{n+1}{4}+1\right)\\
&\geq|C|-9+\frac{3(n+1)}{4}-\frac{n+1}{4}-\frac{1}{2}\\
&=|C|+\frac{n-18}{2}.
\end{align*}	
Thus, for each integer $n\geq19$ we have $|C|>|C|$, which yields a contradiction. This establishes the claim.
\end{proof}	

\noindent{\textsl{Claim~7} :} For each integer $n\geq54$, $|A'|\leq 17$.
\begin{proof}[\tt{Proof of the claim} :]\renewcommand{\qedsymbol}{}
If not, then let $|A'|\geq 18$. We consider the complete graph $G'$ with the vertex set $A'$. Then we have $|V(G')|=|A'|\geq 18$. For each two distinct vertices $a, a'\in V(G')$, we colour the edge $\{a,a'\}$ red if $|T_{a}\cap T_{a'}|\leq1$ and blue if $|T_{a}\cap T_{a'}|\geq2$. Since the Ramsey number $R(4,4)=18$, with respect to this edge-colouring, we get a monochromatic copy of $\mathbb{K}_{4}$ embedded in the complete graph $G'$. That is, there exists a set of four vertices $U\subset A'$ such that either $|T_{a}\cap T_{a'}|\leq1$ for each $a,a'\in U$ with $a\neq a'$ or $|T_{a}\cap T_{a'}|\geq2$ for each $a,a'\in U$ with $a\neq a'$. However, in both cases, we obtain a contradiction from Claim~5 and Claim~6. This establishes the claim.
\end{proof}
 Using Claim~7, we have a decomposition of $A=A'\sqcup A''$, where $A'$ is the same as described before and $A''=A-A'$. Since $|A'|\leq17$, $|A''|\geq |A|-17\geq \frac{3(n+1)}{4}+\frac{t}{2}+\rho-17$. We decompose the vertex set 
\begin{equation*}
V(G)=\{v\}\sqcup A'\sqcup A''\sqcup C''.     
\end{equation*}
Note that $C''=C-\{v\}$ and hence $|C''|=|C|-1= \frac{9}{4}(n+1)-\frac{t}{2}-\rho$. We construct two induced subgraphs $\overline{G''}=\overline{G}[V(G)-(\{v\}\sqcup A')]$ and $G''=G[V(G)-(\{v\}\sqcup A')]$. Note that $\overline{G''}$ forms a complementary subgraph of $G''$.  We consider the bipartition $A''\sqcup C''$ of 
\begin{equation*}
    V(G'')=V(\overline{G''})=V(G)-(\{v\}\sqcup A').
\end{equation*}
Since $v$ is adjacent to each vertex of $A$ in $\overline{G}$, we have for each $w\in A''$, $v\in N_{\overline{G}}(w)$, consequently $|N_{\overline{G''}}(w)|<\frac{3(n+1)}{2}-2$. To prevent the loss of information, we take $|N_{\overline{G''}}(w)|+\frac{1}{2}\leq\frac{3(n+1)}{2}-2$, i.e. $|N_{\overline{G''}}(w)|\leq\frac{3(n+1)}{2}-\frac{5}{2}$. And thus for each $w\in A''$, using \eqref{nbd complementary dependency},
\begin{align*}
|N_{G''}(w)|=&|V(G'')|-1-|N_{\overline{G''}}(w)|\\
 =&|V(G)|-1-|A'|-|N_{\overline{G''}}(w)|\\
 \geq&3(n+1)-|A'|-\frac{3(n+1)}{2}+\frac{5}{2}=\frac{5}{2}+\frac{3(n+1)}{2}-|A'|.
\end{align*}
Since $w$ is adjacent with each of the vertices of $A''$ in $G$, $|N_{G''}(w)\cap A''|=|A''|-1$. Thus we have
\begin{align*}
 |N_{G''}(w)\cap C''|=|N_{G''}(w)|- (|A''|-1)&\geq \frac{5}{2}+\frac{3(n+1)}{2}-|A'|-(|A''|-1)\\
 &=\frac{7}{2}+\frac{3(n+1)}{2}-|A|\\
 &=\frac{7}{2}+\frac{3(n+1)}{4}-\frac{t}{2}-\rho.   
\end{align*}
Using Lemma~\ref{intersection lemma}, for the vertex set bipartition $A''\sqcup C''$ with $d=\frac{7}{2}+\frac{3(n+1)}{4}-\frac{t}{2}-\rho$ for the graph $G''$, we have two distinct vertices $a,b\in A''$ such that 
\begin{align*}
|N_{G''}(a)\cap N_{G''}(b)\cap C''|>\frac{d^{2}}{|C''|}-\frac{d}{|A''|}\geq&\frac{[\frac{3(n+1)}{4}-\frac{t}{2}-\rho+\frac{7}{2}]^{2}}{\frac{9(n+1)}{4}-\frac{t}{2}-\rho}-\frac{\frac{3(n+1)}{4}-\frac{t}{2}-\rho+\frac{7}{2}}{\frac{3(n+1)}{4}+\frac{t}{2}+\rho-17}\\
 &=\frac{1}{4}\frac{[3(n+1)-2t-4\rho+14]^{2}}{9(n+1)-2t-4\rho}-\frac{3(n+1)-2t-4\rho+14}{3(n+1)+2t+4\rho-68}\\
 &=(n+1)-|A|+\frac{P(n,t,\rho)}{Q(n,t,\rho)}\\
 &=\bigg(\frac{n+1}{4}-\frac{t}{2}-\rho\bigg)+\frac{P(n,t,\rho)}{Q(n,t,\rho)},
\end{align*}
where
\begin{align*}
\frac{P(n,t,\rho)}{Q(n,t,\rho)}&=\frac{1}{4}\frac{[3(n+1)-(2t+4\rho)+14]^{2}}{9(n+1)-(2t+4\rho)}-\frac{(n+1)-(2t+4\rho)}{4}-\frac{3(n+1)-(2t+4\rho)+14}{3(n+1)+(2t+4\rho)-68}\\
&=\frac{(21+2t+4\rho)(n+1)-7(2t+4\rho)+49}{9(n+1)-(2t+4\rho)}-\frac{3(n+1)-(2t+4\rho)+14}{3(n+1)+(2t+4\rho)-68}.
\end{align*}

\noindent{\textsl{Claim~8} :} For each integer $n\geq139$, $\frac{P(n,t,\rho)}{Q(n,t,\rho)}>0$ holds.
\begin{proof}[\tt{Proof of the claim} :]\renewcommand{\qedsymbol}{}
Further computation shows 
\begin{align*}
 P(n,t,\rho)&=(36+6t+12\rho)(n+1)^{2}+[(2t+4\rho)^{2}-56(2t+4\rho)-1407](n+1)+\\
 &\hspace{2cm}+539(2t+4\rho)-8(2t+4\rho)^{2}-3332\\
Q(n,t,\rho)&=[9(n+1)-(2t+4\rho)][3(n+1)-68+(2t+4\rho)].
\end{align*}
Here we estimate both $ P(n,t,\rho)$ and $Q(n,t,\rho)$, using $2t+4\rho\in[2t,n+1)$. Note that
\begin{align*}
 P(n,t,\rho)&>24(n+1)^{2}-8(2t+4\rho)^{2}-2191(n+1)-4\cdot 539-3332>16(n+1)^{2}-2191(n+1)-5488,\\
 Q(n,t,\rho)&<9(n+1)[3(n+1)+(2t+4\rho)]<36(n+1)^{2}.
\end{align*}
Thus, for each integer $n\geq139$,
\begin{equation*}
 \frac{P(n,t,\rho)}{Q(n,t,\rho)}>\frac{16(n+1)^{2}-2191(n+1)-5488}{36(n+1)^{2}}>0,  
\end{equation*}
which establishes the claim.
\end{proof}
Here $C''\subset C$ and $G''$ forms a subgraph of $G$ and $\overline{G}[Q]$ is assumed to form a bipartite graph, with bipartition $Q=A\sqcup B$. These imply
\begin{align*}
 N_{G''}(a)\cap N_{G''}(b)\cap C''&\subset N_{G}(a)\cap N_{G}(b)\cap C,\\
 N_{G[Q]}(a)\cap N_{G[Q]}(b)\cap A&=A-\{a,b\}=N_{G}(a)\cap N_{G}(b)\cap A. 
\end{align*}
Thus for two distinct vertices $a,b\in A''$,
\begin{align*}
|N_{G}(a)\cap N_{G}(b)|=&|N_{G}(a)\cap N_{G}(b)\cap A|+|N_{G}(a)\cap N_{G}(b)\cap C|\\
&\geq |N_{G[Q]}(a)\cap N_{G[Q]}(b)\cap A|+|N_{G''}(a)\cap N_{G''}(b)\cap C''|\\
&\geq (|A|-2)+(n+1)-|A|+1=n.
\end{align*}
This contradicts \eqref{K_2n freeness dependency}.
\end{proof}

\begin{remark}
In the proof of Lemma~\ref{nbd conservation 2}, considerable effort is devoted to establish $|A'|\leq17$, which is subsequently used in the estimation of the term $\frac{d}{|A''|}=\frac{d}{|A|-|A'|}\leq\frac{d}{|A|-17}$ involves in the inequality for $\Ex[X]$.  The set $A'$ serves as a \emph{null} set in the sense discussed in the introduction.
\end{remark}

We further invoke a connectivity argument, since one of the sufficient conditions of Lemma~\ref{cycle lemma 1} requires us to show that $\kappa(\overline{G}[Q_{a}])\geq2$, for some $a\in V(\overline{G})$ with 
\begin{equation*}
 |N_{\overline{G}}(a)|\in\left[\frac{3(n+1)}{2}-1, 2(n+1)\right].   
\end{equation*}
Thereafter, showing that it contains a copy of an odd cycle $C_{m}$. Consequently, we complete the proof by obtaining a copy of a wheel $W_{m}$ in $\overline{G}$ with $a$ as a hub. We first establish the following lemma. 

\begin{lemma}\label{connectivity lemma 3}
For each $a\in V(\overline{G})$ with $|N_{\overline{G}}(a)|\in\left[\frac{3(n+1)}{2}-1, 2(n+1)\right]$ and with the decomposition $P_{a}\sqcup Q_{a}=N_{\overline{G}}(a)$, the connectivity parameter of $\overline{G}[Q_{a}]$ is at least $2$, i.e. $\kappa(\overline{G}[Q_{a}])\geq2$.
\end{lemma}
\begin{proof}
If not, let $\kappa(\overline{G}[Q_{a}])\leq 1$. Since $\delta(\overline{G}[Q_{a}])\geq \frac{|N_{\overline{G}}(a)|}{6}-\frac{4}{3}\geq\frac{|Q_{a}|}{6}-\frac{4}{3}\geq \frac{|Q_{a}|}{7}+7$ for sufficiently large $n$, using Lemma~\ref{Star-cycle} the vertex set $Q_{a}$ admits a decomposition  
\begin{equation*}
 Q_{a}= R\sqcup S_{1}\sqcup\ldots\sqcup S_{r},   
\end{equation*}
where $|R|\leq r-1$, for some $r<7$ and $\kappa(\overline{G}[S])\geq 2$ for each $S\in\{S_{1},\ldots,S_{r}\}$. Due to the neighbourhood relation \eqref{nbd dependency}, for each $x\in S$ , where $S\in\{S_{1},\ldots,S_{r}\}$
\begin{equation*} 
N_{\overline{G}[S]}(x)\subset N_{\overline{G}[Q_{a}]}(x)\subset N_{\overline{G}[S]}(x)\sqcup R.    
\end{equation*}
Since each $S\in\{S_{1},\ldots,S_{r}\}$ forms a connected component, thus we have 
\begin{equation*}
 \frac{|Q_{a}|}{7}+7-5\leq\delta(\overline{G}[Q_{a}])-|R|\leq\delta(\overline{G}[S])\leq|S|-1.   
\end{equation*}
This means $|S|\geq3+\frac{|Q|}{7}$ for each $S\in\{S_{1},\ldots,S_{r}\}$.

We choose $B\in\{S_{1},\ldots,S_{r}\}$ such that $|B|=\min\{|S_{1}|,\ldots,|S_{r}|\}$ and $A=(S_{1}\sqcup\ldots\sqcup S_{r})-B$. 

\noindent{\textsl{Claim~1} :} $|A|\in\left[\frac{|V(\overline{H})|}{2}-3,n-1\right]\cap\mathbb{Z}=\left[\frac{3(n+1)}{4}-\frac{7}{2},n-1\right]\cap\mathbb{Z}$.  
\begin{proof}[\tt{Proof of the claim} :]\renewcommand{\qedsymbol}{}
The proof is analogous to that used for the claim in Lemma~\ref{connectivity-first range}, with the modification that we require here is $|R|\leq 5$.
\end{proof}

\noindent So let $|A|=\frac{3}{2}(n+1)-\frac{7}{2}+p$ for some $p\in[0,\frac{n+1}{4}+\frac{3}{2}]$. Then $|B|\geq \frac{3(n+1)}{4}-\frac{7}{2}-p$ as $B= Q_{a}-(A\sqcup R)$, $|Q_{a}|\geq \frac{3}{2}(n+1)-2$ and $|R|\leq 5$. Here we make the following claim.
  
\noindent{\textsl{Claim~2} :} $\overline{G}[B]$ is non-bipartite.
\begin{proof}[\tt{Proof of the claim} :]\renewcommand{\qedsymbol}{}
If possible, then let $\overline{G}[B]$ be bipartite with partition $B_{1}\sqcup B_{2}$ and $|B_{1}|\geq |B_{2}|$. This implies $|B_{1}|\geq \frac{|B|}{2}\geq \frac{3(n+1)}{8}-\frac{7}{4}-\frac{p}{2}$. Observe that any two vertices $x,y \in B_{1}$ are adjacent to all of $(B_{1}\sqcup A)-\{x,y\}$ in $G$. Consequently, 
\begin{equation*}
|N_{G}(x)\cap N_{G}(y)|\geq  |B_{1}|+|A|-2 \geq \frac{9(n+1)}{8}-\frac{31}{4}+\frac{p}{2}\geq\frac{(9n-53)}{8}>n
\end{equation*}
for each $n\geq54$, which contradicts \eqref{K_2n freeness dependency}. This establishes the claim.
\end{proof}

Recall that, in the above decomposition of $Q_{a}$, each $|S|\geq 2$ for each $S\in\{S_{1},\ldots,S_{r}\}$  and consequently $|A|\geq 2$. Also $|B|\leq n-1$, otherwise each pair of distinct vertices in $A$ form a copy of $\mathbb{K}_{2,n}$ in $G$ with the neighbours in $B$. We also have 
\begin{equation*}
    \delta(\overline{G}[B])\geq \delta(\overline{G}[Q_{a}])-|R|\geq \frac{n+1}{4}-\frac{3}{2}-5\geq \frac{n-1}{5}\geq \frac{|B|}{5} 
\end{equation*}
for each $n\geq 121$. Since $\kappa(\overline{G}[B])\geq 2$, we can further apply Lemma~\ref{cycle lemma 1} and Lemma~\ref{cycle lemma 2} on it. Lemma~\ref{cycle lemma 1} ensures that $\mathop{oc}(\overline{G}[B])\geq 2\delta(\overline{G}[B])-1\geq \frac{n+1}{2}-14$. By plugging in $\varepsilon =\frac{1}{5}$ in Lemma~\ref{cycle lemma 2} and using a similar argument as previously, we conclude that $B$ contains a copy of an odd cycle $C_{m}$. 

Recall that, $ K(\frac{1}{10})>K(\frac{1}{5})$. Therefore, 
\begin{equation*}
 m\in \left[K\left(\frac{1}{10}\right),\left(\frac{3-\sqrt{3}}{5}\right)(n+1)-K\left(\frac{1}{10}\right)\right]\subset \left[K\left(\frac{1}{5}\right),\frac{(n+1)}{2}-14-K\left(\frac{1}{5}\right)\right].
\end{equation*}
Hence, the choice of $n$ and $m$ in Construction~\ref{construction} is sufficient here to guarantee the existence of a copy of an odd cycle  $C_{m}$. This odd cycle, together with $v$, forms a copy of a wheel $W_{m}$, which is a contradiction. Therefore, $\kappa(\overline{G}[Q_{a}])\geq 2$.
\end{proof}

The above lemma implies that $\kappa(\overline{G}[Q])\geq 2$ since $|V(\overline{H})|\in\left[\frac{3(n+1)}{2}-1, 2(n+1)\right]$. Finally, it appears to resolve two cases, where the induced graph $\overline{G}[Q]$ forms a bipartite graph or a non-bipartite graph. If the induced subgraph $\overline{G}[Q]$ forms a non-bipartite graph, then Lemma~\ref{connectivity lemma 3} ensures that Lemma~\ref{cycle lemma 1} and Lemma~\ref{cycle lemma 2} apply to the graph $\overline{G}[Q]$. Furthermore,  Corollary~\ref{decomposition lemma 2} implies that $\delta(\overline{G}[Q])\geq \frac{|Q|}{7}$, for sufficiently large integer  $n$. Since $K(\frac{1}{10})>K(\frac{1}{7})$, we conclude in a similar way that $\overline{G}[Q]$ contains a copy of an odd cycle $C_{m}$. This cycle, together with the vertex $v$, forms a copy of a wheel $W_{m}$ in $\overline{G}$, which is a contradiction. This contradicts our assumption that $\overline{G}$ is $W_{m}-$free.

Otherwise, if the induced subgraph $\overline{G}[Q]$ forms a bipartite graph, then using Lemma~\ref{nbd conservation 2}, there exists at least one vertex $x\in Q'$ such that the induced subgraph $\overline{G}[Q_{x}]$ forms a non-bipartite graph. Since 
\begin{equation*}
 |N_{\overline{G}}(x)|\in\left[\frac{3(n+1)}{2}-1,2(n+1)\right],   
\end{equation*}
it mirrors the structural behaviour of  $V(\overline{H})$ in this context. Consequently, we substitute $v$ and $V(\overline{H})$ with $x$ and $N_{\overline{G}}(x)$. By applying a connectivity argument analogous to that used in Lemma~\ref{connectivity lemma 3}, we conclude that $N_{\overline{G}}(x)$ contains a copy of an odd cycle $C_{m}$ within $\overline{G}$. This cycle, together with the vertex $x$, forms a copy of a wheel $W_{m}$ in $\overline{G}$, which is a contradiction. This contradicts our assumption that $\overline{G}$ is $W_{m}-$free.

This concludes the proof of Theorem~\ref{main theorem}.

\begin{acknowledgement}
This research work has no associated data. The authors thank Jie Wang for detecting a typographical error in an earlier version of this draft.     
\end{acknowledgement}

\bibliographystyle{siam}

\end{document}